\documentclass[a4paper,10pt,reqno]{amsart}
\usepackage{thmtools,enumerate,graphicx,xcolor}
\usepackage{etex,float}
\usepackage{ucs}
\usepackage{amsmath, amssymb, amscd, tabu, diagbox}
\usepackage{sseq}
\usepackage[latin1]{inputenc}
\usepackage{amsmath}
\usepackage{amssymb}
\usepackage{epsfig}
\usepackage{psfrag}

\makeatletter
\def\@tocline#1#2#3#4#5#6#7{\relax
  \ifnum #1>\c@tocdepth % then omit
  \else
    \par \addpenalty\@secpenalty\addvspace{#2}%
    \begingroup \hyphenpenalty\@M
    \@ifempty{#4}{%
      \@tempdima\csname r@tocindent\number#1\endcsname\relax
    }{%
      \@tempdima#4\relax
    }%
    \parindent\z@ \leftskip#3\relax \advance\leftskip\@tempdima\relax
    \rightskip\@pnumwidth plus4em \parfillskip-\@pnumwidth
    #5\leavevmode\hskip-\@tempdima
      \ifcase #1
       \or\or \hskip 2em \or \hskip 2em \else \hskip 3em \fi%
      #6\nobreak\relax
    \hfill\hbox to\@pnumwidth{\@tocpagenum{#7}}\par
    \nobreak
    \endgroup
  \fi}
\makeatother

\usepackage[all]{xy}

\makeatletter
\newcommand{\LeftEqNo}{\let\veqno\@@leqno}
\makeatother
\newcommand{\RH}[1]{\makebox[0.7em]{#1}}

\newcommand{\Pfeil}[1]{\makebox[0pt][l]{$\:%
\setlength{\unitlength}{5mm}%
\begin{picture}(0,0)\put(-0.2,0.18){\vector(1,3.5){#1}}\end{picture}$}}%
\newcommand{\khoca}{\texttt{khoca}}
\newlength{\lowerhalftmp}
\newcommand{\lowerhalfx}[1]{\settoheight{\lowerhalftmp}{#1}\addtolength{\lowerhalftmp}{-1.2ex}\raisebox{-.5\lowerhalftmp}{#1}}
\newcommand{\avcfig}[2][]{\lowerhalfx{\includegraphics[#1]{figs/#2.pdf}}}

\newcommand{\xfixed}[2][1cm]{\xrightarrow{\makebox[#1]{$\scriptstyle #2$}}}

\usepackage{microtype}

\usepackage{hyperref}
\usepackage[nameinlink]{cleveref}
\crefname{subsection}{subsection}{subsections}
\Crefname{subsection}{Subsection}{Subsections}
\definecolor{darkblue}{RGB}{0,0,96}
\definecolor{gray}{RGB}{127,127,127}
\definecolor{darkred}{RGB}{160,0,0}
\definecolor{lightyellow}{RGB}{255,255,128}
\hypersetup{colorlinks={true},linkcolor={darkblue},citecolor={darkblue},filecolor={darkblue},urlcolor={darkblue},pdfauthor={Lukas Lewark and Andrew Lobb},pdftitle={New quantum obstructions to sliceness}}

\DeclareMathOperator{\Gr}{Gr}
\DeclareMathOperator{\mult}{mult}

\newcommand{\qua}{\hskip 0.75em plus 0.15em \ignorespaces}
\def\arxiv#1{\relax\ifhmode\unskip\qua\fi
    \href{http://arxiv.org/abs/#1}%
{\tt arXiv:\penalty -100\unskip#1}}    

\def\MR#1{\relax\ifhmode\unskip\qua\fi
    \href{http://www.ams.org/mathscinet-getitem?mr=#1}{MR#1}}
\def\xox#1{\csname xx#1\endcsname}

\newtheorem{theorem}{Theorem}[section]
\newtheorem{definition}[theorem]{Definition}
\newtheorem{lemma}[theorem]{Lemma}
\newtheorem{prop}[theorem]{Proposition}
\newtheorem{proposition}[theorem]{Proposition}
\newtheorem{corollary}[theorem]{Corollary}

\theoremstyle{definition}
\newtheorem{remark}[theorem]{Remark}
\newtheorem{question}[theorem]{Question}

\renewcommand{\epsilon}{\varepsilon}

\DeclareMathAlphabet{\mathpzc}{OT1}{pzc}{m}{it}
\usepackage{mathrsfs}

\newcommand{\Z}{\mathbb{Z}}
\newcommand{\C}{\mathbb{C}}

\newcommand{\R}{\mathbb{R}}

\renewcommand{\qed}{$\hfill \square$ \smallskip \\ }

\newcommand{\tensor}{\otimes}

\newcommand{\hash}{\#}

\renewcommand{\sl}{\mathfrak{sl}}

\newcommand{\F}{\mathcal{F}}

\newcommand{\cC}{\mathcal{C}}

\newcommand{\redH}{\widetilde{H}}
\newcommand{\redC}{\widetilde{C}}
\newcommand{\reds}{\widetilde{s}}

\newcommand{\eqva}{U(n)}

\begin{document}
\thispagestyle{empty}
\title{New quantum obstructions to sliceness}
\author{Lukas Lewark}
\address{Mathematisches Institut\\
Universit\"at Bern\\
Sidlerstrasse 5\\
3012 Bern \\
Switzerland}
\email{lukas@lewark.de}
\urladdr{http://www.lewark.de/lukas/}
\author{Andrew Lobb} 
\address{Mathematical Sciences\\
Durham University\\
South Road\\
DH1 3LE Durham\\
UK}
\email{andrew.lobb@durham.ac.uk}
\urladdr{http://www.maths.dur.ac.uk/users/andrew.lobb/}

\begin {abstract}
It is well-known that generic perturbations of the complex Frobenius algebra used to define Khovanov cohomology each give rise to Rasmussen's concordance invariant $s$.  This gives a concordance homomorphism to the integers and a strong lower bound on the smooth slice genus of a knot.  Similar behavior has been observed in $\sl(n)$ Khovanov-Rozansky cohomology, where a perturbation gives rise to the concordance homomorphisms $s_n$ for each $n \geq 2$, and where we have $s_2 = s$.

We demonstrate that $s_n$ for $n \geq 3$ does not in fact arise generically, and that varying the chosen perturbation gives rise both to new concordance homomorphisms as well as to new sliceness obstructions that are not equivalent to concordance homomorphisms.

\end {abstract}

\maketitle

\tableofcontents

\section{Introduction}
\label{sec:intro}

\subsection{History}
\label{subsec:history}

In \cite{KR1} Khovanov and Rozansky gave a way of associating, for each $n \geq 2$, a finitely generated bigraded complex vector space to a knot $K$.  It arises as the cohomology of a cochain complex

\[ \cdots \rightarrow C^{i-1,j}_{x^n} (D) \rightarrow C^{i,j}_{x^n} (D) \rightarrow C^{i+1,j}_{x^n} (D) \rightarrow \cdots \]

\noindent defined from any diagram $D$ of $K$ which is invariant under Reidemeister moves up to cochain homotopy equivalence.  We write this vector space as

\[ H^{i,j}_{x^n}(K) {\rm ,} \]

\noindent and refer to $i \in \Z$ as the \emph{cohomological} grading and $j \in \Z$ as the \emph{quantum} grading.  This bigraded vector space exhibits as its graded Euler characteristic

\[ \sum_{i,j} (-1)^i q^j {\rm dim}_\C H^{i,j}_{x^n}(K) \]

\noindent the Reshetikhin-Turaev polynomial of $K$ associated to the fundamental irreducible representation of $\sl(n)$.

The reason for the subscript $x^n$ in the notation is that in the definition of $H$ a choice is made of a polynomial $w \in \C[x]$.  Khovanov-Rozansky took $w=x^{n+1}$ as their polynomial but what is important for the definition is really the first derivative of $w$, and that only up to multiplication by a non-zero complex number.  We record this renormalized first derivative in the subscript.

In fact there is a cohomology theory associated to each degree $n$ monic polynomial $\partial w \in \C[x]$ (we write $\partial w$ to remind of us of the connection with the first derivative) which we write as

\[ H^i_{\partial w}(K) {\rm .}\]

\noindent We refer to $\partial w$ as the \emph{potential} of the cohomology theory.  Note that the cohomology theory $H^i_{\partial w}$ keeps a cohomological grading but does not necessarily retain a quantum grading.  However, for any choice of $\partial w$ there is at least a \emph{quantum filtration} on the cohomology:

\[ \cdots \subseteq \F^{j-1}H^i_{\partial w}(K) \subseteq \F^{j}H^i_{\partial w}(K) \subseteq \F^{j+1}H^i_{\partial w}(K) \subseteq \cdots {\rm ,} \]

\noindent arising from a filtration on the cochain complex associated to a diagram

\[ \cdots \rightarrow \F^{j}C^{i-1}_{\partial w} (D) \rightarrow \F^{j}C^{i}_{\partial w} (D) \rightarrow \F^{j}C^{i+1}_{\partial w} (D) \rightarrow \cdots {\rm .}\]

\noindent The filtered cochain-homotopy type of the cochain complex was shown to be an invariant of $K$ by Wu \cite{Wu1}.

We write the bigraded vector space associated to the filtration as

\[ \Gr^j H^i (K) =  \F^{j}H^i_{\partial w}(K) / \F^{j-1}H^i_{\partial w}(K) \rm{.} \]

Gornik was the first to consider a choice of $\partial w$ different from $x^n$, he took $\partial w = x^n -1$.  In \cite{G}, Gornik showed that for any diagram $D$ of a knot, $H_{x^n -1}(D)$ is of dimension $n$ and is supported in cohomological degree $0$ and furthermore he observed that there is spectral sequence with $E_1$ page isomorphic to $H^{i,j}_{x^n}(K)$ and abutting to $\Gr^j H^{i}_{x^n -1}(D)$.  Given a diagram $D$, the $E_0$ page of the spectral sequence can in fact be identified with the standard Khovanov-Rozansky cochain complex

\[ \F^{j}C^{i}_{\partial w} (D) / \F^{j-1}C^{i}_{\partial w} (D) \equiv C^{i,j}_{x^n} (D) {\rm .} \]

This work of Gornik's can be considered a generalization of Lee's result in Khovanov cohomology \cite{Lee} which essentially proved this for the case $n=2$ (in work that predated the definition of Khovanov-Rozansky cohomology).

In works by the second author \cite{Lobb1} and by Wu \cite{Wu1}, this result of Gornik's was generalized to the case where $\partial w$ has $n$ distinct roots.  Furthermore the quantum gradings on the $E_\infty$ pages of the associated spectral sequences were shown to give rise to lower bounds on the smooth slice genus of a knot.

These results should be thought of as a generalization of Rasmussen's seminal work \cite{Ras1}.  This derived from Khovanov cohomology a combinatorial knot invariant $s(K) \in \Z$ and an associated lower bound $|s(K)|$ on the slice genus sufficiently strong to reprove Milnor's conjecture on the slice genus of torus knots (our normalization of $s$ differs from Rasmussen's).  We summarize:

\begin{theorem}[Gornik, Lobb, Wu]
\label{GLWthm}
Suppose $\partial w \in \C[x]$ is a degree $n$ polynomial which is a product of distinct linear factors and $K$ is a knot.  Then there is a spectral sequence, itself a knot invariant, with $E_1$ page $H^{i,j}_{x^n}(K)$ and abutting to $\Gr^j H^{i}_{\partial w}(K)$.

Furthermore $\Gr^j H^{i}_{\partial w}(K)$ is supported in cohomological degree $i=0$ and is of rank $n$.  We can write $j_1(K) \leq j_2(K) \leq \cdots \leq j_n(K)$ so that $\Gr^j H^{i}_{\partial w}(K)$ is isomorphic to the direct sum of $n$ $1$-dimensional vector spaces supported in bidegrees $(0,j_r)$.

If $K_0$ and $K_1$ are two knots connected by a connected knot cobordism of genus $g$ then

\[ | j_r(K_0) - j_r(K_1) | \leq 2(n-1)g \, \, \, {\rm for} \, \, \, 1 \leq r \leq n {\rm .} \]

\noindent It follows from this and knowing the cohomology of the unknot that we must have

\[ g_*(K) \geq \frac{1}{2(n-1)} | j_r(K) - 2r + n + 1 | \, \, \, {\rm for} \, \, \, 1 \leq r \leq n {\rm ,} \]

\noindent where we have written $g_*(K)$ for the slice genus of $K$.
\end{theorem}

The corresponding result in Khovanov cohomology, which can be thought of as the case $n=2$ of Khovanov-Rozansky cohomology, admits a much neater formulation than that of \Cref{GLWthm}.  This is because of work by Mackaay, Turner, and Vaz \cite{MTV} who proved the following:

\begin{theorem}[Mackaay, Turner, Vaz]
\label{MTVthm}
Suppose we are in the situation of \Cref{GLWthm} with $n = 2$.  Then we have that $j_1(K) = 2s(K) - 1$ and $j_2(K) = 2s(K) +1$.
\end{theorem}

It follows that in the case $n=2$ varying $\partial w$ among quadratics with two distinct roots does not change the invariant $\Gr^j H^{i}_{\partial w}(K)$, which is always equivalent to Rasmussen's invariant $s(K)$.

For Gornik's prescient choice of $\partial w$ the second author \cite{Lobb3} showed that a similar `neatness' result holds for general $n$.

\begin{theorem}[Lobb]
\label{Lobb3thm}
Taking $\partial w = x^n -1$ we have that $j_r = 2(n-1)s_n(K) - n + 2r - 1$ for some knot invariant $s_n(K)$.  Furthermore, $s_n$ is a homomorphism from the smooth concordance group of knots to the integers $\frac{1}{n-1} \Z$.
\end{theorem}

\noindent As in the case $n=2$, this \namecref{Lobb3thm} shows that $\Gr^j H^{i}_{x^n - 1}(K)$ is bigraded isomorphic to the cohomology of the unknot but shifted in the quantum direction by an integer $2(n-1)s_n(K)$.

Taken with computations in \cite{Wu1, Lobb1}, \Cref{Lobb3thm} demonstrates that $s_n$ is a \emph{slice-torus} invariant (in that it is a concordance homomorphism and its absolute value provides a bound on the smooth slice genus which furthermore is tight for all torus knots).  This establishes shared properties of $s_n$ with Rasmussen's invariant $s=s_2$ and with the invariant $\tau$ arising from knot Floer homology.
The first author showed that these invariants are not all equal \cite{lew2}, and in fact it seems probable that
$\{\tau, s_2, s_3, \ldots\}$ is an infinite family of linearly independent invariants.

However, the $s_n$ do \emph{not} comprise all slice genus bounds obtainable from separable potentials!  In the light of \Cref{MTVthm} it might be guessed that the integers $j_r(K, \partial w)$ of \Cref{GLWthm} are in fact each equivalent to the single integer $s_n(K)$ in the sense of \Cref{Lobb3thm}.  This guess is wrong.

In fact we shall see that for $n \geq 3$, two different degree $n$ separable potentials can induce different filtrations on the unreduced cohomology.  These filtrations do give rise to slice genus lower bounds, but not in general to concordance homomorphisms (see \Cref{q:quasihom}).  However, a separable potential $\partial w$ and a choice of a root $\alpha$ of that potential gives a \emph{reduced} cohomology theory from which one can extract a slice-torus concordance homomorphism.  In this way we shall recover the classical $s_n$ as well as a host of new invariants.

One may compare the results in this paper with the recent results due to Ozsv{\'a}th-Stipsicz-Szab{\'o} \cite{ossz} in which they determine that varying the filtration on Knot Floer homology gives rise to a number of different concordance homomorphisms.  One may consider their family of homomorphisms to be obtained by varying the slope of a linear function, while ours are obtained by varying all coefficients of a degree $n$ polynomial.

A relatively simple knot exhibiting interesting cohomologies for different choice of potential is the knot $10_{125}$.  We invite the reader to spend the next subsection exploring this knot.

\subsection{An appetizing example}
\label{subsec:an_example}
\begin{figure}[ht!]
	\centering
	\includegraphics[width=90mm]{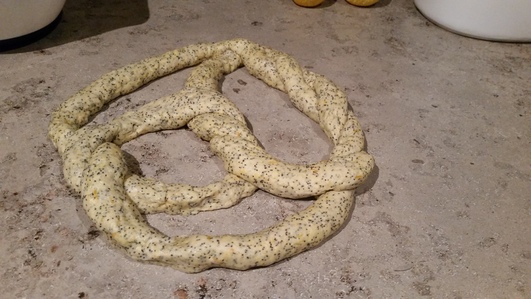}
	\caption{The pretzel knot $P(2,-3,5)$ prior to baking (thanks to Kate Horner and Lauren Scanlon for the image).}
\end{figure}
\begin{figure}[ht]
\avcfig{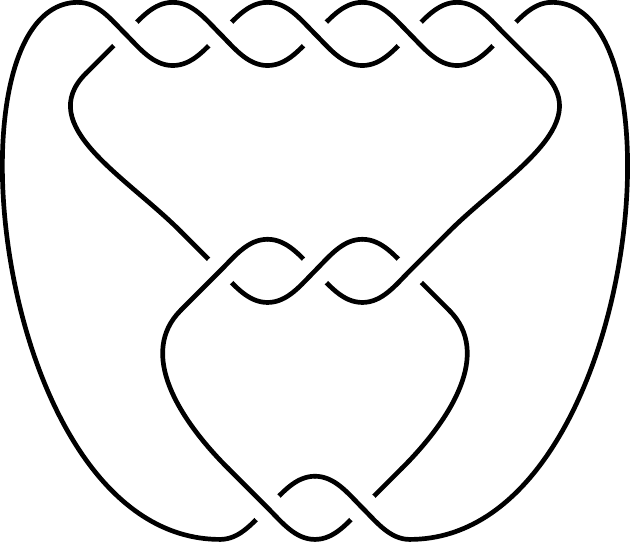}
\caption{A less appetizing diagram of the pretzel knot $P(2,-3,5)$.}
\end{figure}

The pretzel knot $P(2,-3,5)$ appears in the knot table as $10_{125}$, and we shall refer to this knot as $P$ for the remainder of this subsection.

% For some reason, the usual \fcolorbox does not work in the tabu-environment.
\newlength{\oldfboxsep}
\newcommand{\myfcolorbox}[2]{\setlength{\oldfboxsep}{\fboxsep}\setlength{\fboxsep}{0pt}\fbox{\setlength{\fboxsep}{\oldfboxsep}\colorbox{#1}{#2}}}
\newcommand{\Rd}{\myfcolorbox{lightyellow}{1}\,}%
\newcommand{\RR}{\myfcolorbox{lightyellow}{2}\,}%
\newcommand{\Rz}{\phantom{\Rd}}
\begin{table}[ht]%
\centering
\extrarowsep=.57ex%
\begin{tabu}{|[gray]c|*{9}{c|[gray]}}%
\tabucline[gray]-%
\backslashbox{$q$}{$t$} & \RH{$-3$}  & \RH{$-2$}  & \RH{$-1$}  & \RH{0} & \RH{1} & \RH{2} & \RH{3} & \RH{4} & \RH{5} \\\hline
$-9$   &       &       &       &       &       &       &       &       &     1 \\\tabucline[gray]-
$-8$   &       &       &       &       &       &       &       &       &  \Rd  \\\tabucline[gray]-
$-7$   &       &       &       &       &       &       &       &       &       \\\tabucline[gray]-
$-6$   &       &       &       &       &       &       &       &  \Rd\Pfeil{0.7}  &       \\\tabucline[gray]-
$-5$   &       &       &       &       &       &       &     1 &     1 &       \\\tabucline[gray]-
$-4$   &       &       &       &       &       &       & \Rd   &       &       \\\tabucline[gray]-
$-3$   &       &       &       &       &       &     1 &       &       &       \\\tabucline[gray]-
$-2$   &       &       &       &       &       &  \RR \Pfeil{0.7} &       &       &       \\\tabucline[gray]-
$-1$   &       &       &       &       &       &     1 &       &       &       \\\tabucline[gray]-
 0     &       &       &       &       &  \Rd\Pfeil{0.7}  &       &       &       &       \\\tabucline[gray]-
 1     &       &       &       &     2 &     1 &       &       &       &       \\\tabucline[gray]-
 2     &       &       &       &  \RR  &       &       &       &       &       \\\tabucline[gray]-
 3     &       &       &       &     1 &       &       &       &       &       \\\tabucline[gray]-
 4     &       &       &  \Rd\Pfeil{0.7}  &       &       &       &       &       &       \\\tabucline[gray]-
 5     &       &     1 &     1 &       &       &       &       &       &       \\\tabucline[gray]-
 6     &       & \Rd   &       &       &       &       &       &       &       \\\tabucline[gray]-
 7     &       &       &       &       &       &       &       &       &       \\\tabucline[gray]-
 8     &  \Rd\Pfeil{0.7} &       &       &       &       &       &       &       &       \\\tabucline[gray]-
 9     &     1 &       &       &       &       &       &       &       &       \\\tabucline[gray]-
\end{tabu}\\[\baselineskip]
\caption{Unreduced and reduced (the latter printed in yellow boxes) $x^2$-cohomology of the $(2,-3,5)$-pretzel knot.
Non-trivial differentials on the second page of the spectral sequence associated to reduced cohomology with a separable
potential are drawn as arrows.}
\label{fig:pretzel2}%
\end{table}
\begin{table}[ht]%
\centering
\extrarowsep=.57ex%
\begin{tabu}{|[gray]c|*{9}{c|[gray]}}%
\tabucline[gray]-%
\backslashbox{$q$}{$t$} & \RH{$-3$}  & \RH{$-2$}  & \RH{$-1$}  & \RH{0} & \RH{1} & \RH{2} & \RH{3} & \RH{4} & \RH{5} \\\hline
$-18$  &       &       &       &       &       &       &       &       & \Rz 1 \\\tabucline[gray]-
$-16$  &       &       &       &       &       &       &       &       & \Rz 1 \\\tabucline[gray]-
$-14$  &       &       &       &       &       &       & \Rz 1 &       &\Rd  1 \\\tabucline[gray]-
$-12$  &       &       &       &       &       &       & \Rz 1 &       & \Rz 1 \\\tabucline[gray]-
$-10$  &       &       &       &       & \Rz 1 &       &\Rd  1 &       &       \\\tabucline[gray]-
$-8$   &       &       &       &       & \Rz 1 &       & \Rz 1 & \Rz 1 &       \\\tabucline[gray]-
$-6$   &       &       &       &       &\Rd  1 & \Rz 1 &       & \Rd 1 &       \\\tabucline[gray]-
$-4$   &       &       &       &       &       & \Rz 2 &       & \Rz 1 &       \\\tabucline[gray]-
$-2$   &       &       &       & \Rz 2 &       & \RR 2 &       & \Rz 1 &       \\\tabucline[gray]-
 0     &       &       &       & \Rd 3 &       & \Rz 2 &       &       &       \\\tabucline[gray]-
 2     &       & \Rz 1 &       & \RR 3 &       & \Rz 1 &       &       &       \\\tabucline[gray]-
 4     &       & \Rz 1 &       & \Rz 3 &       &       &       &       &       \\\tabucline[gray]-
 6     &       & \Rd 1 &       & \Rz 1 &\Rd  1 &       &       &       &       \\\tabucline[gray]-
 8     &       & \Rz 1 & \Rz 1 &       & \Rz 1 &       &       &       &       \\\tabucline[gray]-
 10    &       &       & \Rd 1 &       & \Rz 1 &       &       &       &       \\\tabucline[gray]-
 12    & \Rz 1 &       & \Rz 1 &       &       &       &       &       &       \\\tabucline[gray]-
 14    &\Rd  1 &       & \Rz 1 &       &       &       &       &       &       \\\tabucline[gray]-
 16    & \Rz 1 &       &       &       &       &       &       &       &       \\\tabucline[gray]-
 18    & \Rz 1 &       &       &       &       &       &       &       &       \\\tabucline[gray]-
\end{tabu}\\[\baselineskip]
\caption{Unreduced and reduced (the latter printed in yellow boxes) $x^5$-cohomology of the $(2,-3,5)$-pretzel knot.}
\label{fig:pretzel5}%
\end{table}

In the \Cref{fig:pretzel2,fig:pretzel5} we give the reduced and unreduced Khovanov-Rozansky cohomologies of $P$ for $n=2$ and $n=5$ (there is no particular reason to choose $5$ over some other integer, but we just want to be explicit).
We encourage the reader to get her hands dirty with a few spectral sequences starting from these cohomologies in order to appreciate something of the phenomena discussed in this paper.

Suppose, for example, that we want to apply \Cref{cor:stdredSS} in order to compute $s_2(P)$ and $s_5(P)$ from the reduced cohomologies.  We are looking for spectral sequences starting from $E_1$-pages the reduced cohomologies of \Cref{fig:pretzel2,fig:pretzel5}, and which have as their final pages 1-dimensional cohomologies supported in cohomological degree $0$.  The differentials on the page $E_i$ increase the cohomological grading by $1$ and decrease the quantum grading by $2i$.

In \Cref{fig:pretzel2} we give the only possible spectral sequence from $E_1 = \redH_{x^2}(K)$ to a 1-dimensional $E_\infty$ page supported in cohomological degree $0$, but in the other Figure the reader will discover two such \emph{a priori} possible $E_\infty$ pages starting from $E_1 = \redH_{x^5}(K)$.

There is better luck to be had in using the unreduced spectral sequences of \Cref{Lobb3thm}.  In the unreduced case the final page is again supported in cohomological degree $0$, but now it is of dimension $n$ (so 2 or 5 in the cases under consideration).  Furthermore, the only non-trivial differentials in the spectral sequence decrease the quantum grading by multiples of $2n$.

From this we can observe that $s_2(P) =1$ and $s_5(P) =1/4$.
The question arises: how far is it accidental that we were unable to compute $s_5(P)$ merely from looking at $\redH_{x^5}(K)$?  It turns out that this failure was inevitable once we determined that $s_5(P)$ is non-integral, as we shall see later in \Cref{subsec:example_redux}.

We ask the reader to return to the unreduced cohomology of \Cref{fig:pretzel5}.  Now look for spectral sequences from this $E_1$ page in which all non-trivial differentials decrease the quantum grading by multiples of $2(n-1) = 8$, and in which the final page is again of dimension $n=5$ supported in cohomological degree $0$.  Whichever spectral sequence of this kind one finds, the final page never has the appearance of a shifted unknot as in \Cref{Lobb3thm}.  Such a spectral sequence would arise from the potential $\partial w = x^5 - x$ (demonstrating for example the non-validity of the \namecref{Lobb3thm} for this new choice of separable potential).

Finally, consider again the reduced cohomology of \Cref{fig:pretzel5}, and look for a spectral sequence in which all non-trivial differentials decrease the quantum grading by multiples of $2(n-1) = 8$ and the final page is of dimension $1$ and is supported in cohomological degree $0$.  There is exactly one such spectral sequence for the knot in question.

In general, given a choice of degree $n$ separable potential $\partial w$ and a root $\alpha$ of that potential, there is a corresponding spectral sequence from reduced $\sl(n)$ cohomology to a 1-dimensional final page supported in cohomological degree $0$.  In this particular case, the spectral sequence corresponds to the separable potential $x^5 - x$ and the choice of root $x=0$.

Furthermore the surviving quantum degree, written as $2(n-1)\widetilde{s}_{x^5 - x,0}(K)$, gives a slice-torus knot invariant $\widetilde{s}_{x^5 - x,0}$ generalizing $s_5$.  Note that for the knot in question we have

\[ \widetilde{s}_{x^5 - x, 0}(P) = 0 \not= \frac{1}{4} = s_5(P) {\rm .} \]

We shall revisit the knot $P$ in \Cref{subsec:example_redux} where we shall shine more light on the concrete phenomena observed above.

\subsection{Summary}
In \Cref{sec:lowerbounds}, we give the definitions and prove the basic properties of
the slice genus lower bounds coming from separable potentials; in particular, it is
shown that not only unreduced, but also reduced Khovanov-Rozansky cohomologies
induce lower slice genus bounds, which are actually more well-behaved than the unreduced bounds:
they are all concordance homomorphisms (in particular slice-torus invariants).
We close by reanalyzing the example of the pretzel knot $P=P(2,-3,5)$ in the light of the properties we have established.
We expect that these results generalize to slice genus bounds for multi-component links, in the appropriate sense;
but for the sake of simplicity we restrict ourselves to knots.

\Cref{subsec:comparison} introduces the notion of KR-equivalent potentials:
potentials inducing homotopy equivalent filtered cochain complexes for all links.
We show that there are at most countably many KR-equivalence classes, and that one of them is generic.
By analyzing the cohomology of the trefoil, we establish that there are at least $n-1$ KR-equivalence classes.

\Cref{subsec:further_egs} exhibits further characteristics of the sliceness obstructions,
which are much more complex than one would have reasonably guessed from what was previously known.

\Cref{sec:calc} discusses the simple form of the cochain complexes of bipartite knots, and the
program \khoca{} (\textbf{k}not \textbf{ho}mology \textbf{ca}lculator) that calculates
their Khovanov-Rozansky cohomologies.

\subsection{Conventions}
For the most part we shall follow the conventions of \cite{KR1}.  These amount to choosing the degree of the variable $x$ to be $2$, and deciding in which cohomological degrees the complex associated to a positive crossing \includegraphics{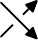} is supported (in degrees $0$ and $1$).  These choices have the consequence that the cohomology of the positive trefoil is supported in non-negative cohomological degrees but in \emph{negative} quantum degrees.  This negative quantum support is in contrast to the situation of the normalization of standard Khovanov cohomology.  Since we encounter Khovanov cohomology only as the case $n=2$ of Khovanov-Rozansky cohomology, we are going to be normalizing the Rasmussen invariant $s=s_2$ so that it is negative on the positive trefoil.

\subsection{Acknowledgments}
The second author wishes to thank Daniel Krasner for bequeathing him bipartite knots and their simple cochain complexes. 
The first author thanks everybody with whom stimulating discussions were had, in particular
Mikhail Khovanov, Paul Turner, Emmanuel Wagner, Andrin Schmidt, and Hitoshi Murakami.
Both authors were funded by EPSRC grant EP/K00591X/1. While working on this paper, the first author was also supported
by the Max Planck Institute for Mathematics Bonn and the SNF grant 137548.

\section{The slice genus lower bounds from separable potentials}
\label{sec:lowerbounds}
\subsection{Reduced cohomology and slice-torus invariants}
\label{sec:red_slice_torus}

Given a knot $K$ with marked diagram $D$, the Khovanov-Rozansky cohomology $H_{\partial w}(K)$ has the structure of a module over the ring $\C[x]/ \partial w$.  In fact, it is the cohomology of a cochain complex $C_{\partial w}(D)$ of free modules over $\C[x]/\partial w$.

This statement is best visualized by cutting the diagram $D$ open at a point marked with the decoration $x$ and thus presenting $D$ as a $(1,1)$-tangle.  Using MOY moves, each cochain group can then be identified with finite sums of quantum-shifted matrix factorizations corresponding to the crossingless $(1,1)$-tangle.  Closing all of these trivial tangles gives the complex associated to the uncut diagram $D$, and each circle now appearing corresponds to a copy of $\C[x]/\partial w$.

This module structure seems at first sight as if it may have some dependence on the choices of diagram and of marked point.  However, if $T$ is a tangle with endpoints labeled $x_1, x_2, \ldots, x_r$, then the Khovanov-Rozansky functor gives a complex of (vectors of) matrix factorizations over the ring $\C[x_1, x_2, \ldots, x_r]$.  Reidemeister moves on $T$ give homotopy equivalent complexes via homotopy equivalences respecting the ground ring.  As a consequence, the $\C[x]/ \partial w$-module structure on the Khovanov-Rozansky cohomology is invariant under Reidemeister moves performed on the (1,1)-tangle.  Finally it is an exercise for the reader to see that if $D$ and $D'$ are two Reidemeister-equivalent diagrams, (each with a marked point on corresponding link components), then $D$ and $D'$ can be connected by a sequence of Reidemeister moves that take place away from the marked points and which take the marked point of $D$ to that of $D'$.

In the case of standard Khovanov-Rozansky cohomology with $\partial w = x^n$, the action of $x$ on the cochain complex preserves the cohomological grading and raises the quantum grading by $2$.  For explicitness we make a definition.

\begin{definition}
We define the reduced Khovanov-Rozansky cohomology $\redH_{x^n}(K)$ of a knot $K$ to be the cohomology of the cochain complex $(x^{n-1}) C_{x^n}(D) [1 - n]$, where the closed brackets denote a shift in quantum filtration.
\end{definition}

The reduced Khovanov-Rozansky cohomology $\redH_{x^n}$ has as its graded Euler characteristic the Reshetikhin-Turaev $\sl(n)$ polynomial normalized so that the unknot is assigned the polynomial $1 \in \Z[q^{\pm 1}]$.

\begin{remark}
We note that in the literature the reduced Khovanov cohomology, for example, is often defined in such a way that its graded Euler characteristic is the Jones polynomial with a surprising normalization: the unknot is assigned polynomial $q^{-1} \in \Z[q^{\pm 1}]$.  We consider our convention to be possibly a little more natural.
\end{remark}

We now wish to give a good definition for a reduced Khovanov-Rozansky cohomology of a knot $K$ using a separable potential $\partial w$, i.e. a potential that is the product of distinct linear factors:

\[ \partial w = \prod_{i=1}^n (x - \alpha_i) {\rm .}\]

\noindent For any marked diagram $D$ of $K$, $H_{\partial w}(K)$ is the cohomology of a cochain complex $C_{\partial w}(D)$ of free $(\C [x] / \partial w)$-modules, inducing a $(\C [x] / \partial w)$-module structure on $H_{\partial w}(K)$.  In fact, we know that $H_{\partial w}(K)$ is $n$-dimensional and that the action of $x$ on $H_{\partial w}(K)$ splits the cohomology into $n$ $1$-dimensional eigenspaces with eigenvalues $\alpha_1, \alpha_2, \ldots, \alpha_n$.  In other words, $H_{\partial w}(K)$ is a free rank $1$ module over the ring $\C [x] / \partial w$.  The reader should note, however, that the quantum filtration of $H_{\partial w}(K)$ need not correspond to an overall shift of the usual filtration on $\C[x] / \partial w$.

\begin{definition}
Suppose $\alpha$ is a root of the degree $n$ monic separable polynomial $\partial w$.  We define $\redH_{\partial w, \alpha}(K)$ (the $(\partial w, \alpha)$\emph{-reduced cohomology} of the knot $K$ with marked diagram $D$) to be the cohomology of the cochain complex

\[ \redC_{\partial w, \alpha}(D) := \left( \frac{\partial w}{x - \alpha} \right) C_{\partial w}(D) [1-n] {\rm ,} \]

\noindent where the square brackets denote a shift in the quantum filtration.

\end{definition}

First note that $\redH_{\partial w, \alpha}(K)$ is certainly a knot invariant, this follows from a similar, but not totally isomorphic, discussion to that appearing at the start of this section:  the cochain complex $\redC_{\partial w, \alpha}(D)$ is a quantum-shifted subcomplex of $C_{\partial w}(D)$.  If $D$ and $D'$ are marked-Reidemeister-equivalent marked diagrams (and the Reidemeister moves take place away from the marked point) then $C_{\partial w}(D)$ and $C_{\partial w}(D')$ are cochain homotopy equivalent $\C[x] / \partial w$-cochain complexes (where the $x$ corresponds to the marked point).  The homotopy equivalences can then be restricted to the subcomplexes $\redC_{\partial w, \alpha}(D)$ and $\redC_{\partial w, \alpha}(D')$ with no modification (this is really an argument for a general ring $R$ about subcomplexes of $R$-complexes given by the action of an ideal of $R$).

We note that one should \emph{not} expect, in general, that $\redH_{\partial w, \alpha}(K)$ is filtered-isomorphic to $\left( \frac{\partial w}{x - \alpha} \right) H_{\partial w}(K) [1-n]$ (which is also a knot invariant).  Indeed, we shall see examples where it certainly differs.

The shift in the quantum degree is to ensure that $\redH_{\partial w, \alpha}(U)$ has Poincar\'e polynomial $1$ for $U$ the unknot and any choice of $(\partial w, \alpha)$.  We shall show

\begin{theorem}
\label{thm:nonstdredSS}
For any knot $K$ and for each separable choice of $(\partial w, \alpha)$, the reduced cohomology $\redH_{\partial w, \alpha}(K)$ is 1-dimensional.  Furthermore, there exists a spectral sequence with $E_1$-page $\redH^{i,j}_{x^n}(K)$ and $E_\infty$-page $\Gr^j \redH^i_{\partial w, \alpha}(K)$.
\end{theorem}

Taking Gornik's choice of potential we obtain a corollary.

\begin{corollary}
\label{cor:stdredSS}
For any knot $K$ there exists a spectral sequence with $E_1$-page $\redH_{x^n}(K)$ such that the $E_\infty$-page is 1-dimensional and has Poincar\'e polynomial $q^{2(n-1)s_n(K)} \in \Z[q,t]$.
\end{corollary}

This corollary is not surprising, and may even be considered `folklore', but as far as we know there is no proof in the literature.

Since the reduced cohomology for a separable potential is always $1$-dimensional we can make another definition.

\begin{definition}
	For a knot $K$ and $(\partial w, \alpha)$ as above, we define the $(\partial w, \alpha)$ \emph{reduced slice genus bound} $\reds_{\partial w, \alpha} (K) \in \frac{1}{(n-1)}\Z$ to be $1/2(n-1)$ times the $j$-grading of the support of the 1-dimensional vector space $\Gr^j\redH^0_{\partial w, \alpha}(K)$
\end{definition}

We have taken this choice of normalization so that we have

\[ \reds_{\partial w, \alpha} (T(2,3)) = s_n (T(2,3)) = -1 {\rm ,} \]

\noindent for any choice of $n$ and of $(\partial w, \alpha)$ where we write $T(2,3)$ for the positive trefoil.

\begin{definition}[{cf. \cite{livingston,lew2}}]
Let $S: \cC \rightarrow \R$ be a homomorphism from the smooth concordance group of oriented knots to the reals.  We say that $S$ is a \emph{slice-torus} invariant if

\begin{enumerate}
	\item $g_*(K) \geq |S(K)|$ for all oriented knots $K$, where we write $g_*(K)$ for the smooth slice genus of $K$.
	\item $S(T(p,q)) = \frac{-(p-1)(q-1)}{2}$ for $T(p,q)$ the $(p,q)$-torus knot.
\end{enumerate}
\end{definition}

\begin{theorem}
	\label{thm:redconc}
Suppose $\alpha$ is a root of the degree $n$ monic separable polynomial $\partial w$.  Then we have that $\reds_{\partial w, \alpha}$ defines a map

\[  \reds_{\partial w, \alpha} : \mathcal{C} \longrightarrow \frac{1}{2(n-1)}\Z \]

\noindent which is a slice-torus invariant.
\end{theorem}

Before proving Theorems \ref{thm:nonstdredSS} and \ref{thm:redconc}, we remind the reader of the work of Gornik's \cite{G} which established cocycle representatives for cohomology with a separable potential.  (In fact Gornik considered only the potential $\partial w = x^n - 1$, but his arguments apply to all separable potentials without critical change.)

Fix any separable $\partial w$ with roots $\alpha_1, \ldots , \alpha_n$, and let $\Gamma$ be a MOY graph.  The cohomology of $\Gamma$, which we shall write simply as $h_{\partial w}(\Gamma)$, is a filtered complex vector space.  If $\Gamma$ occurs as a resolution of some link diagram $D$, then $h_{\partial w}(\Gamma)$ appears on a corner of the Khovanov-Rozansky cube as a cochain group summand of $C_{\partial w}(D)$.

A basis for $h_{\partial w}(\Gamma)$ is given by all \emph{admissible decorations} of $\Gamma$, that is, all decorations of the thin edges of $\Gamma$ with roots of $\partial w$ satisfying the admissibility condition.  This condition is the requirement that at each thick edge two distinct roots decorate each entering thin edge, and the same two roots decorate the exiting thin edges.

If we let $\Gamma$ vary over all resolutions of a diagram $D$ we thus obtain a basis for each cochain group $C_{\partial w}^i(D)$.  By considering how the Khovanov-Rozansky differential acts on the bases, Gornik showed that a basis for the cohomology $H_{\partial w}(D)$ is given by cocycles corresponding to decorations that arise in the following way.  Starting with a decoration by roots of $\partial w$ of the components of $D$, take the oriented resolution at crossings where the roots agree and the thick-edge resolution at crossings where the roots differ.  For each such decoration of components of $D$ by roots, this produces a resolution $\Gamma$ and a cocycle in $h_{\partial w}(\Gamma) \subseteq C_{\partial w}(D)$ surviving to cohomology.  In the case that $D$ is a diagram of a ($1$-component) knot, it follows that Gornik's cocycle representatives live in the summand of the cohomological degree $0$ cochain group corresponding to the oriented resolution of $D$.

Essentially, Gornik's argument proceeded by Gauss elimination, grouping other basis elements into canceling pairs and leaving only the generating cocycles described above.

\begin{lemma}[Gauss Elimination as in \cite{fastcompu}]
\label{lem:gauss}
In an additive category with an isomorphism $h$, the cochain complex
\[
\xymatrix@=1.2cm{
\ldots \ar[r] & P \ar[r]^-{\left(\begin{smallmatrix} * \\ g \end{smallmatrix}\right)} &  Q\oplus R
\ar[r]^-{\left(\begin{smallmatrix} h & i \\ j & k\end{smallmatrix}\right)} & S\oplus T
\ar[r]^-{\left(\begin{smallmatrix} * & \ell\end{smallmatrix}\right)} & U \ar[r] & \ldots
}
\]
is homotopy equivalent to
\[
\xymatrix@=1.7cm{
\ldots \ar[r] & P \ar[r]^{g} &  R
\ar[r]^{k\, -\, j h^{-1} i} & T
\ar[r]^{\ell} & U\ar[r] & \ldots.
}
\]
\end{lemma}

More explicitly, we describe the situation in the case of knots.  Let $D$ be any knot diagram, and let $O(D)$ be the oriented resolution of $D$.  We define

\[ h_{\partial w}(O(D)) : = \Big(\bigotimes_c \C [x_c] / \partial w (x_c) [1-n] \Big) [(1-n)w(D)] \]

\noindent where the tensor product is taken over all components $c$ of $O(D)$ and $w(D)$ stands for the writhe of $D$.  Now $h_{\partial w}(O(D))$ is naturally a summand of the cochain group $C_{\partial w}^0(D)$.
Let

\[ \partial w = \prod_{i=1}^n (x - \alpha_i) {\rm ,} \]

\noindent then for any knot diagram $D$ there exist $n$ linearly independent cocycles

\[ g^D_{\alpha_i} \in h_{\partial w}(O(D)) \subseteq   C^0_{\partial w}(D) \quad {\rm for} \, \, 1 \leq i \leq n \]

\noindent given by

\[ g^D_{\alpha_i} = \bigotimes_c \frac{1}{\partial w '(\alpha_i)} \left( \frac{\partial w (x_c)}{x_c - \alpha_i}  \right) .\]

\begin{theorem}[Gornik \cite{G}]
These cocycles $g^D_{\alpha_i}$ descend to give a basis for the cohomology $H^0_{\partial w}(D)$.  Furthermore, there is a spectral sequence with $E_1$-page $H^{i,j}_{x^n}(D)$ abutting to $Gr^j H^i_{\partial w}(D)$.
\end{theorem}

Gornik's proof of the existence of the spectral sequence relied on identifying the $E_0$-page of the spectral sequence associated to the filtered complex $C_{\partial w}(D)$ with the Khovanov-Rozansky cochain complex $C_{x^n}(D)$.

Now we have enough to prove the theorems stated earlier.

\begin{proof}[Proof of {\protect\Cref{thm:nonstdredSS}}]
	First note that Gornik's proof that the unreduced cohomology $H_{\partial w}(K)$ is $n$-dimensional also demonstrates that the reduced $\redH_{\partial w, \alpha}$ cohomology is $1$-dimensional.  To see this, observe that the reduced cohomology is the cohomology of a subcomplex $\redC_{\partial w, \alpha}(D)$ of $C_{\partial w}(D)[1-n]$ where $D$ is a marked diagram for $K$.  The subcomplex is spanned by exactly $\frac{1}{n}$ of Gornik's generators for $C_{\partial w}(D)[1-n]$ (those with the decoration $\alpha$ at the marked thin edge).  These generators can be Gauss-eliminated following Gornik's recipe, leaving just one cocycle
	
	\[ g^D_\alpha \in \redC_{\partial w, \alpha}(D) \subseteq C_{\partial w}(D) [1-n] \]
		
	\noindent which generates the cohomology.
	
	To show the existence of the spectral sequence we wish to see that the $E_0$-page associated to $\redC_{\partial w, \alpha}(D)$ corresponds exactly to the complex $x^{n-1}C_{x^n}(D)$ which computes standard reduced cohomology.  The invariance of the spectral sequence under choice of diagram is automatic: the spectral sequence in question is just that associated to the filtered complex $\redC_{\partial w, \alpha}(D)$, whose filtered homotopy type we know is independent of the diagram.
	
	We shall proceed by using the restriction of the correspondence between the $E_0$-page of the spectral sequence associated to the filtered complex $C_{\partial w}(D)$ with the Khovanov-Rozansky cochain complex $C_{x^n}(D)$.
	
	Let $\Gamma$ be a resolution of $D$, inheriting the marked point of $D$.  Then $h_{x^n}(\Gamma)$ is a $\C[x]/x^n$-module, and $h_{\partial w}(\Gamma)$ is a $\C[x]/\partial w$-module.  To see the module structures, perform MOY decompositions away from the marked point until you arrive at a single marked circle.  This gives the module structures explicitly as
	
	\[ h_{x^n}(\Gamma) = \bigoplus_{i} (\C[x]/x^n) [a_i] \]
	
	\noindent and
	
	\[ h_{\partial w}(\Gamma)= \bigoplus_{i} (\C[x]/\partial w) [a_i]\]
	
	\noindent for a finite sequence of integers $\{ a_i \}$ where the first decomposition is as graded modules, and the second as filtered modules.
	
	The conclusion we want is almost clear from this, we just need to say a bit more about the module structures.
	
	From the definition by matrix factorizations, $h_{\partial w}(\Gamma)$ is the cohomology of a 2-periodic complex of free $\C[x]$-modules where the modules are graded and the differentials are filtered.  On the other hand, $h_{x^n}(\Gamma)$ arises as the cohomology of the 2-periodic complex with the same cochain groups but where just the top-degree components of the differential are retained.
		
	For any $\psi \in h_{\partial w}(\Gamma)$, write $\psi' \in h_{x^n}(\Gamma)$ for the associated graded element.  Then observe that if $q(x)$ is a polynomial in $x$ with leading term $x^r$, we have $(q(x) \psi)' = x^r \psi'$ if and only if they are of the same grading.
		
	Now it is clear that the associated graded vector space to $\frac{\partial w}{x - \alpha} h_{\partial w}(\Gamma)$ is exactly $x^{n-1} h_{x^n}(\Gamma)$.\end{proof}

Now we prove the second theorem.

\begin{proof}[Proof of \protect\Cref{thm:redconc}]
	Firstly we show that $|\reds_{\partial w, \alpha}(K)|$ is a lower bound for the slice genus of $K$.  We make use of the arguments already given in the unreduced case by Lobb and Wu, which we briefly summarize here.
	
	To each Morse move (otherwise known as handle attachment), from a diagram $D$ to a diagram $D'$, there is associated a cochain map $C_{\partial w}(D) \rightarrow C_{\partial w}(D')$.  This cochain map is filtered of degree $n-1$ in the case of a 1-handle attachment, and of degree $1-n$ for 0- and 2-handle attachments.  Taking these together with the homotopy equivalences between Reidemeister-equivalent cochain complexes gives a way to associate a filtered cochain map to a representation of a link cobordism.  Specifically, given a movie $M$ of a cobordism between diagrams $D^0$ and $D^1$, by composing the cochain maps we already have, we can thus associate a filtered cochain map
	
	\[ M_* : C_{\partial w}(D^0) \rightarrow C_{\partial w}(D^1) {\rm .} \]
	
	\noindent Adding up the contributions from the various maps, we observe that this map is filtered of degree $(1-n) \chi(M)$ where we write $\chi(M)$ for the Euler characteristic of the surface represented by $M$.
	
	Finally one shows that if $M$ is a movie of a connected cobordism between the two knot diagrams $D^0$ and $D^1$, then $M_*(g_{D^0_\alpha})$ is a non-zero multiple of $g_{D^1_\alpha}$ for $\alpha$ any root of $\partial w$.  Hence, we see that $M_*$ induces an isomorphism on cohomology.
	
	The slice genus bound statements in \Cref{GLWthm} follow immediately.
	
	For the reduced statement, we note that the cochain maps on the unreduced cochain complexes induced by handle moves and Reidemeister moves on marked diagrams respect the $\C[x]/\partial x$-structure of the cochain groups.  Hence by restriction they also induce maps on the reduced cochain complexes.
	
	We now take a marked movie $M$ (which can be thought of as describing a cobordism together with an embedded arc without critical points) between the marked diagrams $D^0$ and $D^1$.  This induces a map
	
	\[ M_* : \redC_{\partial w, \alpha}(D^0) \rightarrow \redC_{\partial w, \alpha}(D^1) {\rm ,} \]
	
	\noindent again of filtered degree $(1-n) \chi(M)$.  Note furthermore that the generator of the cohomology is preserved since this map is just a restriction of the unreduced map which does preserve the generator.
	
	This dispenses with the question of whether $|\reds_{\partial w, \alpha}(K)|$ is a lower bound for the slice genus of $K$.  That this bound is tight for torus knots (and in fact for all positive knots), is immediate from consideration of a positive diagram.
	
	Finally we show that $\reds_{\partial w, \alpha}$ is a concordance homomorphism.  From the 0-crossing diagram of the unknot $U$ we know that $\reds_{\partial w, \alpha}(U) = 0$, it therefore remains to show that $\reds_{\partial w, \alpha}(K_1 \# K_2) = \reds_{\partial w, \alpha}(K_1) + \reds_{\partial w, \alpha}(K_2)$, where we write $\#$ for the connect-sum operation.
	
	Let then $D_1$ and $D_2$ be two marked diagrams, and let $D = D_1 \# D_2$ be the marked diagram formed by the connect sum, with the connect sum taking place at the marked point.
	
	We write $\Phi$ for the map 
	
	\[ \Phi : C_{\partial w}(D_1) \otimes_{\C} C_{\partial w}(D_2) \rightarrow C_{\partial w}(D) \]
	
	\noindent induced by 1-handle addition to the diagram $D_1 \sqcup D_2$.  Let $\Gamma_1$ and $\Gamma_2$ be MOY resolutions of $D_1$ and $D_2$ respectively, and let $\Gamma = \Gamma_1 \# \Gamma_2$ be the corresponding resolution of $D$.
	
	By repeated MOY simplification away from the marked points we can reduce $\Gamma_1 \sqcup \Gamma_2$ to the disjoint union of two marked circles $U_1$ and $U_2$.  Performing the same MOY simplification to $\Gamma$ we can reduce to the marked circle $U_1 \# U_2$.  Thus for $\{a_i\}$ being a finite sequence of integer shifts we see that $\Phi$ restricted to the cochain group summand $S : = h_{\partial w}(\Gamma_1) \otimes_{\C} h_{\partial w}(\Gamma_2)$ is the map
	
	\[ \Phi|_S : \bigoplus_{i} \left( \C[x_1 , x_2]/(\partial w(x_1), \partial w(x_2)) \right) [a_i] \rightarrow \bigoplus_i \left( \C[x]/\partial w(x)  \right) [a_i] \]
	
	\noindent given by `multiplication' or, in other words, the identification $x_1 = x_2 = x$.
	
	Restricting $\Phi$ to the shifted subcomplex $\redC_{\partial w, \alpha}(D_1) \otimes \redC_{\partial w, \alpha}(D_2)$, the restriction to the corresponding summand
	
	\[ \widetilde{S} : = \frac{\partial w(x_1)}{x_1-\alpha} h_{\partial w}(\Gamma_1) \otimes \frac{\partial w(x_2)}{x_2-\alpha} h_{\partial w}(\Gamma_2) [2-2n] \]
	
	\noindent is projectively the map
	
	\begin{eqnarray*}
	\Phi |_{S'} &:& \bigoplus_{i} \frac{\partial w(x_1)}{x_1 - \alpha}\frac{\partial w(x_2)}{x_2 - \alpha}\left( \C[x_1 , x_2]/(\partial w(x_1), \partial w(x_2)) \right) [2 - 2n + a_i] \\ &\longrightarrow& \bigoplus_i \frac{\partial w(x)}{x - \alpha}\left( \C[x]/\partial w(x)  \right) [1 - n + a_i]
	\end{eqnarray*}
	
	\noindent again given by multiplication.
	
	This is a filtered degree $0$ isomorphism of vector spaces with filtered degree $0$ inverse, and hence $\Phi$ restricts to an isomorphism of filtered cochain complexes $\redC_{\partial w, \alpha}(D_1) \otimes \redC_{\partial w, \alpha}(D_2)$ and $\redC_{\partial w, \alpha}(D)$, and we are done.
	
\end{proof}

To deduce \Cref{cor:stdredSS} we first prove some results of independent interest that relate unreduced cohomology to the reduced concordance homomorphisms.

\begin{prop}
\label{prop:red<>unred}
Let $D$ be a diagram of the knot $K$, let $\alpha$ be a root of $\partial w$, and let $g_\alpha \in C^0_{\partial w}(D)$ be the Gornik cocycle corresponding to $\alpha$.

Then the filtration degree $q$ of $[g_\alpha] \in H^0_{\partial w}(D)$ satisfies

\[ q \leq s_{\partial w, \alpha}(K) + n - 1 {\rm .} \]
\end{prop}

\begin{proof}
The degree $q$ is the smallest filtration degree among elements of the coset $g_\alpha + d(C^{-1}(D))$ (where we write $d$ for the differential in the Khovanov-Rozansky cochain complex).  But the reduced cohomology also has $g_\alpha$ as a cocycle representative.  From this it follows that $s_{\partial w, \alpha}(K) + n - 1$ is the smallest filtration degree among elements of the coset $g_\alpha + [\partial w / (x - \alpha)] d(C^{-1}(D))$.  This latter coset is a subset of the former coset, from which the result follows.
\end{proof}

We can also compare the filtration on the entire unreduced cohomology with the collection of $n$ slice genus bounds corresponding to each root of $\partial w$.  The next proposition implies in particular that the bounds arising from unreduced cohomology and the unreduced bounds differ by at most 1.
\begin{proposition}\label{prop:relation red unred}
Let $K$ be a knot. We have $H_{\partial w}(K) = q^{j_1(K)} + \ldots + q^{j_n(K)}$ with
$j_1(K)\leq \ldots\leq j_n(K)$ as in \Cref{GLWthm}.
Sort the roots $\alpha_1, \ldots, \alpha_n$ of $\partial w$ such that
$\widetilde{s}_{\partial w, \alpha_1}(K) \leq \ldots \leq
\widetilde{s}_{\partial w, \alpha_n}(K)$.
Then we have 
\[
\left|j_i(K) - 2(n-1)\widetilde{s}_{\partial w, \alpha_i}(K)\right| \leq n-1.
\]
\end{proposition}

\begin{proof}
Multiplication by $\partial w / (x-\alpha_i)$ gives a filtered map of degree $2n-2$
\[
C_{\partial w}(D) \to \frac{\partial w}{x-\alpha_i} C_{\partial w}(D),
\]
in other words a filtered map
\[
C_{\partial w}(D) \to \widetilde{C}_{\partial w, \alpha_i}(D)[1-n].
\]
Summing over $i$ yields a filtered map
\[
C_{\partial w}(D) \to \bigoplus_{i=1}^n \widetilde{C}_{\partial w, \alpha_i}(D)[1-n].
\]
This induces a bijective filtered map on cohomology (the inverse is not necessarily filtered, too)
\[
H_{\partial w}(K) \to \bigoplus_{i=1}^n \widetilde{H}_{\partial w,\alpha_i}(K)[1-n],
\]
since it takes generating cocycles to generating cocycles.  Hence we get for each $i\in\{1,\ldots,n\}$:
\[
j_i(K) \geq 2(n-1)\widetilde{s}_{\partial w, \alpha_i}(K) + 1 - n \]
or equivalently
\[ 2(n-1)\widetilde{s}_{\partial w, \alpha_i}(K) - j_i(K) \leq n - 1 {\rm .}
\]
To complete the proof, one could now resort to taking the mirror image of $D$.
But in fact, it suffices to consider the inclusion map
\[
\frac{\partial w}{x-\alpha_i} C_{\partial w}(D) \to C_{\partial w}(D),
\]
which is filtered, and sum over $i$ to produce a bijective filtered map
\[
\bigoplus_{i=1}^n \widetilde{C}_{\partial w, \alpha_i}(D)[n-1]
\to
C_{\partial w}(D),
\]
which also gives a bijective filtered map on cohomology.  Thus we have
\[
2(n-1)\widetilde{s}_{\partial w, \alpha_i}(K) + n - 1 \geq j_i(K) \quad\implies\quad
j_i(K) - 2(n-1)\widetilde{s}_{\partial w, \alpha_i}(K) \leq n - 1.
\]
and the proof is complete.
\end{proof}

Finally it is now a simple matter to deduce \Cref{cor:stdredSS}.

\begin{proof}[Proof of {\protect\Cref{cor:stdredSS}}]
	The homomorphism $\tilde{s}_{x^n - 1, \alpha}$ is independent of the choice of root $\alpha$ since there exist linear maps of $\C$ which cyclically permute the roots.  By \Cref{prop:relation red unred}, it follows that $2(n-1)\tilde{s}_{x^n - 1, \alpha}(K)$ is a value at distance at most $n-1$ from each of $j_1(K), \ldots, j_n(K)$.  Since $j_n(K) - j_1(K) = 2(n-1)$, this determines $\tilde{s}_{x^n - 1, \alpha}(K)$ completely and we have
	
	\[ \tilde{s}_{x^n - 1, \alpha}(K) = \frac{j_n(K) - j_1(K)}{4(n-1)} = s_n(K) {\rm .} \] 
\end{proof}

\subsection{Unreduced cohomology}
\label{subsec:basic_unred}
In this subsection we consider the unreduced theory which is in a sense richer than the reduced theory.  We still obtain slice genus lower bounds, but in general we give up the property of defining a concordance homomorphism, although we shall be able to define a concordance quasi-homomorphism.

We fix a separable potential $\partial w$ and recall the definition of the integers $j_i (K)$ from \Cref{GLWthm} describing the filtration on $H_{\partial w}(K)$. Now, since the complex associated to the mirror image of a diagram is the dual complex,
the invariants are still well-behaved with respect to the mirror image, that is to say:
\begin{prop}\label{prop:mirror}
For all $i\in\{1,\ldots,n\}$ we have
\[
j_i(\overline{K}) = -j_{n-i}(K).
\]
\end{prop}
However, the filtration on the unreduced cohomology $H_{\partial w} (K_1 \# K_2)$ is not in general determined by those on $H_{\partial w}(K_1)$ and on $H_{\partial w}(K_2)$, see \Cref{q:sumdet}. Still, some bounds can be given:

\begin{prop}
\label{prop:quasihomomorphisms}
For knots $K_1$ and $K_2$ we have
\begin{eqnarray*}
j_1(K_1) + j_1(K_2) + 1 - n \leq j_i (K_1 \# K_2) \leq j_n(K_1) + j_n(K_2) - 1 + n
\end{eqnarray*}
\end{prop}

\begin{proof}
The $1$-handle cobordism $K_1 \sqcup K_2 \rightarrow K_1 \# K_2$ induces a surjection on unreduced cohomology (this is part of the proof that unreduced cohomology gives slice genus lower bounds) which is filtered of degree $n-1$.  Furthermore we have the isomorphism as filtered vector spaces $H_{\partial w}(K_1 \sqcup K_2) = H_{\partial w} (K_1) \otimes H_{\partial w}(K_2)$.

Hence we have $j_i(K_1 \# K_2) + 1 - n \leq j_n(K_1 \# K_2) + 1 - n \leq j_n(K_1) + j_n(K_2)$.

The other inequality follows from the same argument applied to the mirrors of $K_1$ and $K_2$.
\end{proof}

Such boundedness results suggest that one should at least be able to extract from unreduced cohomology \emph{quasi-homomorphisms} from the knot concordance group to the reals.  For example, one could make the definition

\[ s_{\partial w} (K) : = \frac{j_1(K) + \cdots + j_n(K)}{2n(n-1)} {\rm .} \]

\noindent The absolute value of $s_{\partial w}$ certainly gives a lower bound on the slice genus which is tight for torus knots (since $s_{\partial w}$ is the average of $n$ functions with these properties), and furthermore one has

\begin{prop}
\label{prop:quasihomeg}
The function $s_{\partial w}$ is a quasi-homomorphism.
\end{prop}

\begin{proof}
\begin{eqnarray*}
 &      & | s_{\partial w}(K_1 \# K_2) - s_{\partial w} (K_1) - s_{\partial w}(K_2) | \\
 & =    & \left| \frac{1}{2n(n-1)}\sum_{i = 1}^n (j_i(K_1 \# K_2) - j_i(K_1) - j_i(K_2))\right| \\
 & \leq & \frac{1}{2n(n-1)} \bigg( 3n(n-1) + \sum_{\substack{\alpha \, : \\ \partial w(\alpha) = 0}} 2(n-1)| \tilde{s}_{\alpha}(K_1 \# K_2) - \tilde{s}_{\alpha}(K_1) - \tilde{s}_{\alpha}(K_2)| \bigg)\\
 & =    & \frac{3}{2}
\end{eqnarray*}

\noindent directly from \Cref{prop:relation red unred}.
\end{proof}

We now give a definition which will enable us to be briefer in the sequel.

\begin{definition}
We put a partial order on Laurent polynomials in $q$ with non-negative integer coefficients by writing $F_1(q) \geq F_2(q)$ if and only if $F_1(q) - F_2(q)$ is expressible as a sum of polynomials

\[ F_1(q) - F_2(q) = \sum_i q^{u_i} - q^{v_i} \]

\noindent where $u_i \geq v_i$ for all $i$.
\end{definition}

\begin{lemma}
\label{positivesum}
Let $\partial w$ be a separable potential, $K$ a knot, and $P$ a positive knot.  Then we have
\[ \Gr^j H^i_{\partial w} (K \# P) = \Gr^{j} H^i_{\partial w} (K)[2(n-1)s_n(P)] , \]
where the square brackets denote a shift in the quantum grading.  In other words, taking connect sum with a positive knot $P$ has the effect of an overall shift equal to the genus of $P$.
\end{lemma}

\begin{proof}

Writing

\[ F_L(q) = \sum_j {\rm dim}(\Gr^{j} H^0_{\partial w}(L))q^j \]

\noindent for the Poincar\'e polynomial of a knot $L$, we will be done if we can show that

\[ F_{K \hash P}(q) = q^{2(n-1)s_n (P)}F_K(q) {\rm .} \]

First note that since $P$ is positive we have that $s_n(P)$ is non-positive and furthermore $-s_n(P)$ is equal to the slice genus of $P$.  Hence there is a genus $-s_n(P)$ cobordism from $K$ to $K \hash P$,  and we can conclude from \Cref{GLWthm} that

\[ F_{K \# P}(q) \geq q^{2(n-1)s_n (P)}F_K(q) {\rm .} \]

To establish the reverse inequality, let $D_K$, $D_P$, and $D_{K \hash P}$  be a diagram for $K$, a positive diagram for $P$, and the diagram of $K \hash P$ formed by a 1-handle addition between $D_K$ and $D_P$, respectively.

Now consider the cochain complex $C^i_{\partial w}(D_P)$.  The cochain complex is supported in non-negative homological degrees, and the cohomology of the cochain complex is supported in degree $0$.  It follows that if $g \in C^0_{\partial w}(D_P)$ is a cocycle, then the filtration grading of $[g] \in H^0_{\partial w}(D_P)$ agrees with the filtration grading of $g$.

Let $U_+$ be a positive diagram of the unknot whose oriented resolution $O(U_+)$ has the same number of components as $O(D_P)$.  We know that we have a filtered isomorphism of $\C[x]/\partial w(x)$-modules:

\[ H^0_{\partial w}(U_+) \equiv (\C [x] / \partial w (x)) [1-n] {\rm .} \]

The argument above tells us that we can identify the cohomology of $D_P$ with that of $U_+$ up to an overall shift, hence we have

\begin{eqnarray*} 
H^0_{\partial w}(D_P) &\equiv& (\C [x] / \partial w (x)) [(n-1)(\vert O(D_P) \vert - w(D_P) - 1) ] \\
&\equiv& (\C [x] / \partial w (x)) [1 - n + 2(n-1) s_n (P)]
\end{eqnarray*}

\noindent where we write $\vert O(D_P) \vert$ for the number of components of $O(D_P)$.

Under this identification, each generator $[g^{D_P}_i] \in H^0_{\partial w}(D_P)$ corresponds to a non-zero multiple of an $\alpha_i$-eigenvector of the action of $x$, in other words, to a non-zero multiple of the element

\[ \frac{\partial w(x)}{x - \alpha_i} \in \C [x] / \partial w (x)  [1 - n + 2(n-1) s_n (P)] {\rm .} \]

Now note that we have

\[ 1 = \sum_i \frac{1}{\prod_{j \not = i} (\alpha_i - \alpha_j)} \frac{\partial w(x)}{x - \alpha_i} \in \C [x] / \partial w (x)  [1 - n + 2(n-1) s_n (P)] {\rm ,} \]

\noindent so that we see that there is a cocycle $h \in C^0(D_P)$ such that the filtration grading of $[h] \in H^0_{\partial w}(D_P)$ agrees with that of $h$ and is $1-n + 2(n-1) s_n (P)$.  Furthermore, note that $h$ is a linear combination of the generators $g^{D_P}_i$ with each coefficient non-zero.

There is a map

\[ \Phi : H^0_{\partial w}(K) \otimes H^0_{\partial w}(P) \rightarrow H^0_{\partial w}(K \hash P)\]

\noindent induced by the 1-handle addition Morse move from $D_K \sqcup D_P$ to $D_{K \hash P}$.  This map is filtered of degree $n-1$.  If $0 \not= k \in H^0_{\partial w}(K)$ then we have that $0 \not= \Phi ( k \tensor [h]) \in H^0_{\partial w}(K \hash P)$, since there exists a cocycle representative for $k$ expressible as a linear combination of the generators $g^{D_K}_i$ and $h$ is a linear combination of the generators $g^{D_P}_i$ with each coefficient non-zero.

Writing ${\rm gr}$ for the filtration grading we see that

\[ {\rm gr} (\Phi ( k \tensor [h])) \leq {\rm gr}(k) + {\rm gr}([h]) + n -1 = {\rm gr}(k) + 2(n-1)s_n(P) {\rm ,} \]
and this completes the proof.
\end{proof}

Along with the usual mirror argument establishing the corresponding result for connect sum with negative knots, this is enough to deduce that the $j_i(K)$ share many properties of slice-torus invariants.  Some of these properties are described in \cite{lobbineq} by the second author.  The arguments there are specific to the situation of Khovanov cohomology, but there are topological proofs due to unpublished work by Kawamura and \cite{lew2} by the first author.

We summarize the structure of the topological arguments and how they apply in the situation of a separable potential $\partial w$.  Given a diagram $D$, one constructs cobordisms to positive and negative diagrams $D^+$ and $D^-$ respectively.  By \Cref{GLWthm}, one establishes lower bounds on $j_i(D)$ from the first cobordism and upper bounds on $j_i(D)$ from the second.  If one can make good choices for $D$, $D^+$, and $D^-$, then one can make the upper and lower bounds agree, determining each $j_i(D)$ completely.  

In such a good situation, because the cohomologies $H_{\partial w}(D^+)$ and $H_{\partial w}(D^-)$ are just shifted copies of the cohomologies of the unknot, it follows that $H_{\partial w}(D)$ is also a shifted version of the unknot.  Note that nowhere in this process have we relied on the particular choice of $\partial w$.  Hence we have an isomorphism as filtered vector spaces $H_{\partial w}(D) \equiv H_{x^n - 1}(D)$.  Furthermore, in such a good situation, the same argument applies in the reduced case so that $\tilde{s}_{\partial w, \alpha} = s_n$ for any choice of root $\alpha$.

Moreover, one can take the obvious cobordisms between diagrams $D' \# D$ and $D' \# D^+$ and between $D' \# D$ and $D' \# D^-$.  In the case of a good situation as above, it follows from the resulting inequalities that $H_{\partial w}(D' \# D)$ is a shift of $H_{\partial w}(D')$ by $2(n-1)s_n(D)$.

We give some known classes of knots which have such a good situation:

\begin{theorem}
\label{thm:homogeneousetc}
If $\partial w$ is separable and $\alpha$ is a root of $\partial w$, $K$ is a  quasi-positive, quasi-negative, or homogeneous knot (included in these categories are positive, negative, and alternating knots, but not all quasi-alternating knots) and $K'$ is a knot, we have that

\begin{itemize}
\item $\tilde{s}_{\partial w, \alpha}(K) = s_n(K)$,
\item $H_{\partial w}(K) \equiv H_{x^n - 1}(K)$ as filtered vector spaces,
\item $H_{\partial w}(K' \# K) \equiv H_{\partial w}(K')[2(n-1) s_n(K)]$. \qed
\end{itemize}
\end{theorem}

There is an observation exploited by Livingstone \cite{livingston} that says if $K^+$ and $K^-$ are knots related by a crossing change, then there is a genus $1$ cobordism between $K^+$ and $K^- \# T_{2,3}$ where we write $T_{2,3}$ for the positive trefoil (this is specific example of the general observation that two intersection points of opposite sign in a connected knot cobordism can be exchanged for a single piece of genus).  The resulting inequality gives immediately

\begin{prop}
\label{prop:crossingchange}
If $K_+$ and $K_-$ are knots with diagrams that are related by a crossing change, then
\[
0 \leq j_i(K_-) - j_i(K_+) \leq 2(n-1) {\rm .}
\] \qed
\end{prop}

\subsection{Appetizing example revisited}
\label{subsec:example_redux}

We return to our example of \Cref{subsec:an_example}, the pretzel knot $P = 10_{125}$, and reanalyze its cohomology in light of what we now know.  We start with an easy proposition in which we do not require that our potential $\partial w$ is separable.

\begin{definition}
A page $E_i$ (where $i\geq 1$) of a spectral sequence $E$ is called \emph{significant} if
it is not isomorphic to $E_{i-1}$ as a doubly graded vector space. Otherwise, it is called \emph{insignificant}.
\end{definition}

\begin{prop}\label{prop:mod}
Let $\partial w = \sum_i a_i x^i\in\mathbb{C}[x]$ be a potential of degree $n > 0$, and $K$ a knot.
Suppose that  $\exists m \geq 1: i\not\equiv n \pmod{m} \implies a_i = 0$.
Then there is a $\Z/2m\Z$-grading on the cohomology $H_{\partial w}(K)$ which is respected by the spectral sequence.  In particular, all pages $E_i$ of the spectral sequence arising from the filtered homotopy class of complexes corresponding to $K$ and the potential $\partial w$
with $i \not\equiv 1 \pmod{m}$ are insignificant.
\end{prop}

\begin{proof}
If $D$ is a diagram of $K$, note that the differential of $C_{\partial w}(D)$ preserves the filtration
degree modulo $2m$, thus giving a $\Z/2m\Z$-grading on the cohomology.  Splitting the complex along this cyclic grading, we see that the differentials of the spectral sequence must also respect the grading.
The differential on the $i$-th page has $q$-degree $2i$, and thus is $0$ if
$m$ does not divide $i$.
\end{proof}

We can use the cyclic grading on the cohomology to deduce consequences which are manifested in the example of the knot $P$ and the spectral sequences that we analyzed in \Cref{subsec:an_example}.  For example:

\begin{theorem}
\label{prop:integraltildesn}
The concordance homomorphism $\tilde{s}_{x^n - x, 0}$ factors through the integers.
\end{theorem}

Observe that if $s_n(K)$ is not an integer, then we therefore have $s_n(K) \not= \tilde{s}_{x^n - x, 0}(K)$.  We note that this proposition implies the existence of two distinct spectral sequences from $\redH_{x^n}(K)$ abutting to $1$-dimensional $E_\infty$ pages supported in cohomological degree $0$, specifically one way these spectral sequences differ is in the quantum grading of the support of their $E_\infty$ pages.

In fact in this situation the unreduced cohomology must also change with the potential.

\begin{prop}
\label{prop:fractionalsn}
Suppose $K$ is such that

\[ s_n(K) \in \frac{1}{n-1}\Z \setminus \Z {\rm .} \]

Then as a filtered vector space we have

\[ H_{x^n - x} (K) \not\equiv H_{x^n - 1} (K) {\rm .}\]
\end{prop}

\begin{proof}[{Proof of \protect\Cref{prop:integraltildesn} and \protect\Cref{prop:fractionalsn}}]
First observe that the potential $x^n - x$ is of the form considered in \Cref{prop:mod}, so that the cohomology is $\Z/2(n-1)\Z$-graded.

The roots of the potential are $0, \xi^0, \xi^1, \xi^2, \ldots, \xi^{n-2}$ where $\xi = e^{2\pi i/(n-1)}$.  If we are given a diagram $D$, then corresponding to these roots are the Gornik cocycles which generate the cohomology -- we shall write these as $g, g_0, g_1, \ldots, g_{n-2}$.

Each Gornik generator is an element of the cochain group summand $S$ corresponding the oriented resolution of $D$.  If the oriented resolution of $D$ has $c$ components then, ignoring the overall quantum shift, the corresponding cochain group summand is the vector space

\[ S = \frac{\C[x_1]}{x_1^n -x_1} \otimes \cdots \otimes \frac{\C[x_c]}{x_c^n - x_c} {\rm .} \]

The generators $g_i$ are given by

\[ g_i = \frac{(x_1^n - x_1) \cdots (x_c^n - x_c)}{(x_1 - \xi^i) \cdots (x_c - \xi^i)} {\rm .}\]

We claim that the vector space $G = \langle g_0, \ldots, g_{n-2} \rangle$ has a basis $h_0, h_1, \ldots, h_{n-2}$ where $h_i$ is a homogeneous polynomial of degree $2i$ with respect to the cyclic grading $\Z/2(n-1)\Z$ inherited from the usual $\Z$-grading on $\C[x_1, \ldots, x_c]$.  The proof is a straightforward check and an explicit argument is given \emph{mutatis mutandis} in Lemma 2.4 of \cite{Lobb3}.

Putting back in the overall quantum shift, it follows that the associated $\Z/2(n-1)\Z$-graded vector space to the 1-codimensional subspace
\[
\langle [g_0], \ldots, [g_{n-2}] \rangle \subset H_{x^n - x}(D)
\]
is 1-dimensional in each even grading if $n$ is odd and in each odd grading if $n$ is even.  Furthermore, using \Cref{reducedsymmetry} we see that $s_{x^n - x, \xi^i}$ is independent of choice of $i$.  This implies via \Cref{prop:relation red unred} that the associated graded vector space to the 1-codimensional subspace above has Poincar\'{e} polynomial $q^r(1 + q^2 + \cdots + q^{2n-4})$ for some integer $r$.

The generator $g$ is homogeneous of grading $n-1$ with respect to the $\Z / 2(n-1)\Z$-grading on the cochain complex.  By the definition of $\tilde{s}_{x^n - x, 0}$, it follows immediately that $\tilde{s}_{x^n - x, 0}$ is always integral.

Moreover, since the cyclic grading $n-1$ is 2-dimensional in unreduced cohomology, we see that if the Poincar\'{e} polynomial of the associated graded vector space to $H_{x^n - x}(K)$ is of the form $q^{2r}(q^{1-n} + q^{3-n} + \cdots + q^{n-1})$ for some integer $r$, then we must have that $2(n-1)$ divides $r$.

On the other hand, the Poincar\'{e} polynomial of the associated graded vector space to $H_{x^n - 1}(K)$ is exactly $q^{2(n-1)s_n(K)}(q^{1-n} + q^{3-n} + \cdots + q^{n-1})$.  Hence if $s_n(K)$ is not an integer, we must have

\[ H_{x^n - x}(K) \not\equiv H_{x^n - 1}(K) {\rm .} \]
\end{proof}

\section{Comparing different potentials}
\label{subsec:comparison}
\subsection{The KR-equivalence classes}

\begin{definition}
We call two potentials $\partial w$ and $\partial w'$
\emph{KR-equivalent over a link $L$},
denoted by $\partial w \sim_L \partial w'$,
if
$C_{\partial w}$ and $C_{\partial w'}$ are cochain homotopy equivalent over $\mathbb{C}$.
Furthermore, we call those potentials \emph{KR-equivalent}, denoted by $\partial w \sim \partial w'$,
if they are KR-equivalent over all links.
\end{definition}
This section is devoted to investigating the space of KR-equivalence classes.  In this paper we restrict ourselves to the unreduced case, but the reduced case is also interesting.  In particular, note that unreduced cohomologies with KR-equivalent potentials are filtered isomorphic,
but the corresponding reduced cohomologies need not be.

Throughout this section, we will frequently use the following graded rings:
\begin{align*}
R_n = \mathbb{C}[a_0, \ldots, a_{n-1}],\quad & \deg a_i = 2(n-i), \\
R_n[x],\quad & \deg x = 2, \\
R_n[x]/p,\quad & p = x^n + a_{n-1}x^{n-1} + \ldots + a_0.
\end{align*}
\begin{theorem}[\cite{krasnerEquivariant}, cf. also \cite{wuequivariant}]
There is an \emph{equivariant $\sl_n$-cohomology} theory as follows:
to a marked diagram $D$ of a link $L$, a finite-dimensional graded cochain complex $C_{\eqva}(D)$ of free $R_n[x]/p$-modules is associated,
such that complexes of equivalent marked diagrams are homotopy equivalent over $R_n[x]/p$.
We will denote its cohomology by $H_{\eqva}(L)$.
Evaluating by $e: R_n\to\mathbb{C}$ gives a
filtered cochain complex of free $\mathbb{C}[x]/\partial w$-modules
with $\partial w = x^n + e(a_{n-1})x^{n-1} + \ldots + e(a_0)$,
that is the usual $\sl_n$-complex with potential $\partial w$.
\end{theorem}
Equivariant cohomology is in a sense a \emph{universal} $\mathfrak{sl}_n$-homology, from which
the reduced and unreduced cohomologies and spectral sequences for all potentials of degree $n$ can be recovered.
\begin{proposition}\label{prop:linear_equivalence}
Suppose that $D$ is knot diagram, $\partial w_1$ is a degree $n$ potential, and that we define another degree $n$ potential $\partial w_2$ by

\[ \partial w_2 (x) = \frac{1}{a^n} \partial w_1 (ax + b) {\rm ,} \]
where $a,b \in \C$ with $a \not= 0$.  Then:
\begin{enumerate}\renewcommand{\theenumi}{\roman{enumi}}
\item There is a filtered cochain homotopy equivalence $\varphi: C_{\partial w_1}(D) \to C_{ \partial w_2}(D)$.
\item %
$\varphi(x \cdot c) = (ax + b)\cdot \varphi(c)$.
\end{enumerate}
\end{proposition}
\begin{proof}
The first part is due to Wu, \cite[Proposition 1.4]{genericwu}. The second part follows immediately from the construction of $\varphi$.
\end{proof}
This implies that every potential is KR-equivalent to a potential whose $x^{n-1}$-coefficient is zero. Another corollary is the following:
\begin{corollary}\label{reducedsymmetry}
Let $\alpha_1, \ldots, \alpha_n$ be the roots of $\partial w_1$. Then
\[
\widetilde{s}_{\partial w_1(x), \alpha_i}(K) = \widetilde{s}_{\partial w_2, a^{-1}(\alpha_i - b)}(K).
\]
\end{corollary}
So far, particular attention has been focused on potentials of the following kind:
\begin{definition}
We call $\partial w$ a \emph{Gornik potential} if for some $\beta, \gamma \in \C$ with $\gamma \not= 0$ we have

\[ \partial w = (x - \beta)^n - \gamma {\rm .} \]
\end{definition}
From the previous proposition, we get:
\begin{corollary}
\label{gornikniceness}
Suppose that $K$ is a knot, $\partial w$ is a Gornik potential and $\alpha$ is any root of $\partial w$, then we have
\begin{enumerate}
\item $H_{\partial w}(K) \equiv H_{x^n - 1}(K)$ as a filtered vector space,
\item $\reds_{\partial w, \alpha} (K) = s_n (K)$.
\end{enumerate}
\end{corollary}
Let us give a geometric interpretation of the situation for separable potentials:
a separable potential can be given by its set of roots in the complex plane.
If two such sets are related by an affine symmetry of the plane then their corresponding potentials are KR-equivalent.
In particular, there is only one KR-equivalence class of potentials of degree 2,
as was first proved in \cite{MTV} -- all potentials of degree 2 are Gornik.
For higher degrees, however, the situation is more complicated. For
the main result of this section, identify the set of polynomials of degree $n$ with $\mathbb{C}^n$ and endow it with the Zariski topology.
\begin{theorem}\label{thm:gen}
For a fixed link $L$ and a fixed $n$, there are only finitely many KR-equivalence classes of
polynomials of degree $n$ over $L$.
One of these classes is generic in the sense that all other classes
are finite unions of intersections of an Zariski-open set with a Zariski-closed set that is not $\mathbb{C}^n$.
\end{theorem}
\begin{corollary}
Let $n \geq 1$ be fixed.  There are at most countably many classes of KR-equivalence.
One of these classes is generic in the sense that it contains a countable intersection of
non-empty Zarisiki-open (and thus dense) sets.
\end{corollary}
This notion of genericity is strong enough for example to imply that the complement of the generic
class has measure zero.
At the moment, it is not clear whether for a fixed $n$,
there is in fact an infinity of KR-equivalence classes;
in the next \Cref{subsec:bunch} we will see that there are at least $n - 1$.

In the proof, we use the strategy of \emph{successive Gauss elimination}
as described in \cite{heddni} to compute the spectral sequence. Let us briefly explain
this strategy:
our additive category of choice is finite-dimensional filtered vector spaces over a field.
Gauss elimination may be used to dispose of all isomorphisms of a cochain complex $C_0$, yielding a
homotopy equivalent cochain complex $C_0'$ whose differentials on the 0-th page of cohomology
are trivial, i.e. $E_0(C_0') = E_1(C_0')$. So it is possible to define a filtered complex $C_1$
as regrading of $C_0'$ by shifting the filtration degree of the $t$-th cohomology group down by $t$.
We have $E_k(C_1) = E_{k+1}(C_0')$. Now repeat this procedure -- let $C_1'$ be homotopy equivalent
to $C_1$ with trivial differentials on the first page, $C_2$ its regrading etc.
At some point $C_{\ell}$ will have trivial differentials, and at that point all the pages of
the spectral sequence have been computed.

On the one hand, this gives a practical algorithm to compute a spectral sequence; indeed it
is this algorithm that we use in our program \khoca, see \Cref{subsec:compu}. On the other hand,
it establishes that doing so determines the filtered homotopy type:
\begin{proposition}
Two finite-dimensional filtered cochain complexes over a field with respective spectral sequences
$E$ and $E'$ are homotopy equivalent if and only if $E_i$ and $E'_i$ are
isomorphic doubly graded vector spaces for all $i \geq 1$.
\end{proposition}
\begin{figure}[t]
\includegraphics[scale=0.25]{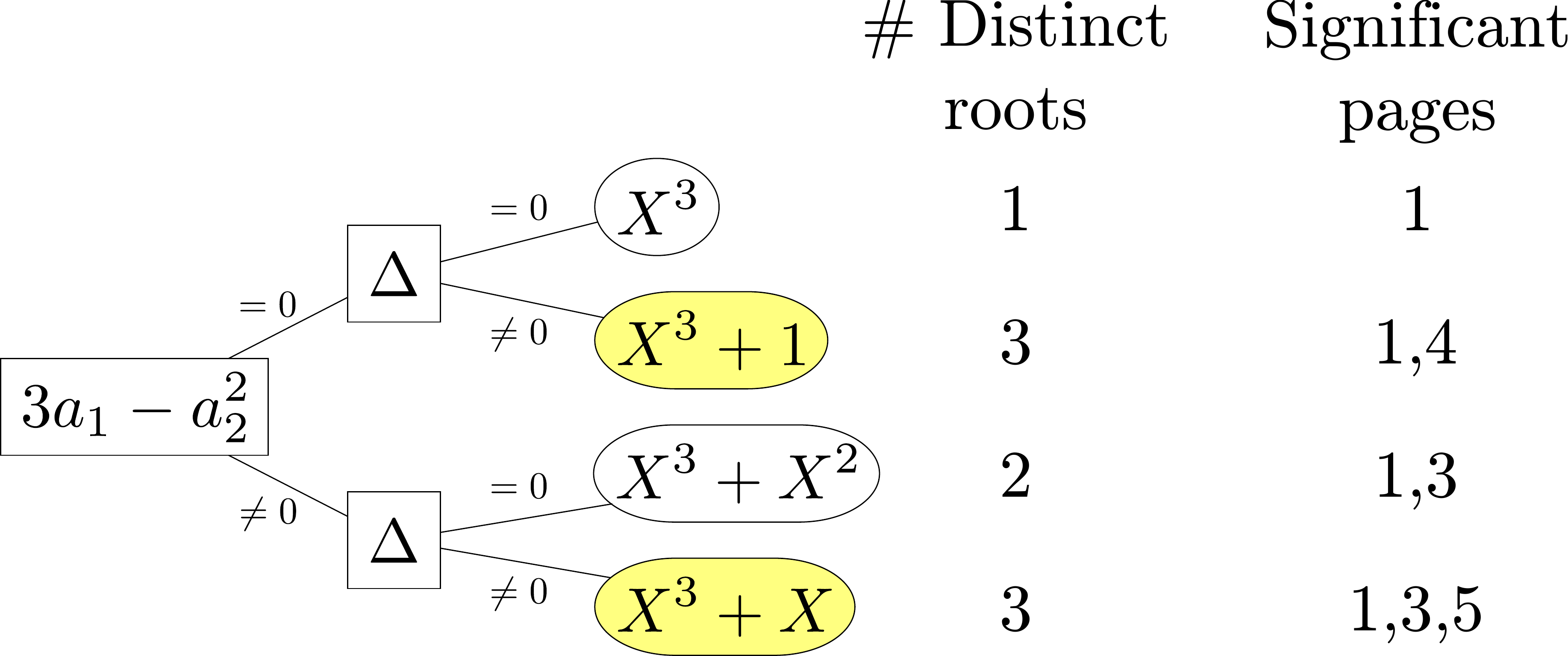}
\caption{A tree illustrating the KR-equivalence classes of the trefoil for $n = 3$. Polynomials in rectangular boxes are those on whose
vanishing the class depends, as in the proof of {\Cref{thm:gen}}; $\Delta$ denotes the discriminant. The round boxes mark representative potentials of the corresponding KR-equivalence class. Separable potentials are in yellow boxes. The generic class is at the bottom.}
\label{fig:tree3}
\end{figure}
\begin{proof}[Proof of {\Cref{thm:gen}}]
Forgetting $x$, the equivariant $\sl_n$-cochain complex $C_{\eqva}$ is a complex of free graded $R_n$-modules of finite rank,
where the $a_i$ carry a non-negative degree. That is in fact all the information
we need on $C_{\eqva}$; consider any cochain complex $C$ with these properties and a subset $U\subset\mathbb{C}^n$.
Then $U$ is divided into equivalence classes of vectors $(v_0, \ldots, v_{n-1})$,
on which evaluating by $a_i\mapsto v_i$ induces homotopy equivalent filtered cochain complexes.
Let us prove for all pairs $(C,U)$ that $U$ decomposes as disjoint union of finitely many sets $D_i$,
such that each equivalence class is the union of some $D_i$'s, and such that for each $D_i$,
one may select two sets $T_i, T_i' \subset R_n$ such that
$D_i = U \cap Z(T_i) \cap Z(T_i')^c$, where $Z(T) = \bigcap_{p\in T} p^{-1}(0)$.
Moreover, for at most one $i$ we have $Z(T_i) = \mathbb{C}^n$.
This implies the statement of the \namecref{thm:gen}.

We proceed by induction.
The statement is obviously true for a complex with trivial differentials.
Otherwise, assume that the statement holds for all pairs $(C', U')$ such
that either $C'$ has smaller total dimension than $C$; or has equal total dimension,
but fewer non-zero matrix entries.
If all degree-preserving differentials of $C$ are zero, regrade following the description of
successive Gauss elimination above.
Then pick a non-zero degree-preserving matrix entry $p$ of $C$ and consider $(C, U \cap p^{-1}(0))$.
Without changing the equivalence classes, $p$ may be replaced by zero,
so by the induction hypothesis $U \cap p^{-1}(0) = \bigsqcup_{i=1}^k D_i$,
where
\[
D_i \ =\  U \cap p^{-1}(0) \cap Z(T_i) \cap Z(T_i')^c \ =\  U \cap Z(T_i \cup \{p\}) \cap Z(T_i')^c.
\]
On the other hand, consider $(C, U \cap (p^{-1}(0))^c)$.
The polynomial $p$ does not vanish for any evaluation;
so let us perform a Gauss elimination on $p$, and multiply
the matrix with $p$ afterwards. The ensuing complex $C'$
has the same equivalence classes as $C$, because its
evaluation at any point in $U \cap (p^{-1}(0))^c$ is homotopy
equivalent to $C$. Moreover, $C'$ has smaller total dimension,
so by the induction hypothesis,
$U \cap (p^{-1}(0))^c = \bigsqcup_{j=1}^{\ell} \widetilde{D}_j$, where
\[
\widetilde{D}_j \ =\ U \cap (p^{-1}(0))^c \cap Z(\widetilde{T}_j) \cap Z(\widetilde{T}_j')^c\ =\ U \cap Z(\widetilde{T}_j) \cap Z(\widetilde{T}_j'\cup\{p\})^c.
\]
Note that $Z(T_i\cup\{p\})$ is a proper subset of $\mathbb{C}^n$ for all $i$, and at most
one of the $Z(\widetilde{T}_j\cup\{p\})$ is not. So 
\[
U = \bigsqcup_i D_i \sqcup \bigsqcup_j \widetilde{D}_j
\]
is the decomposition of $U$ whose existence was to be proven.
\end{proof}
\begin{figure}[t]
\includegraphics[scale=0.25]{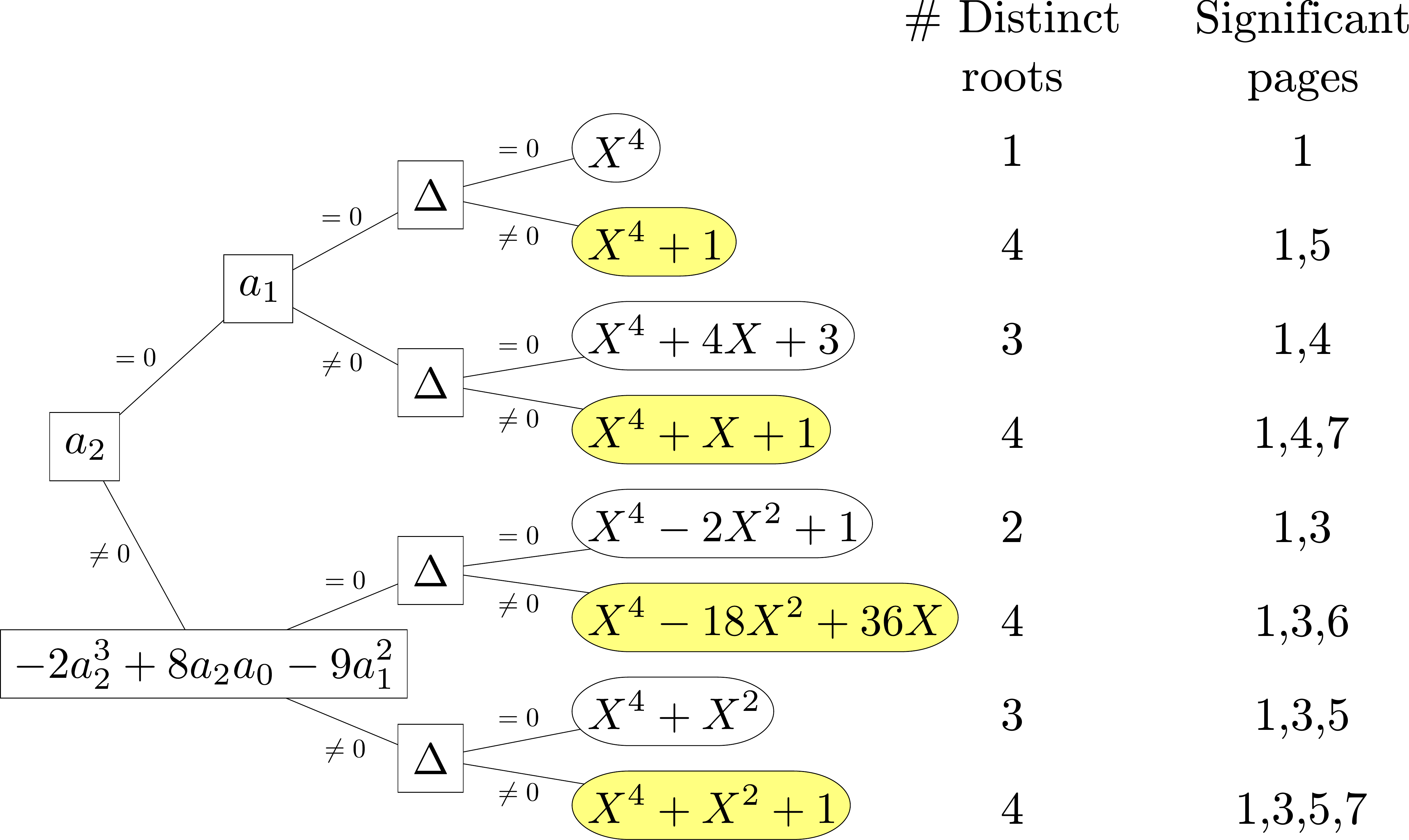}
\caption{A tree as in \Cref{fig:tree3} for the trefoil and $n = 4$. For simplicity, we have set $a_3 = 0$ from the beginning;
note that every polynomial is KR-equivalent to another with $a_3 = 0$ by \Cref{prop:linear_equivalence}.}
\end{figure}

\subsection{A lower bound on the number of KR-equivalence classes}
\label{subsec:bunch}
\begin{theorem}\label{cor:existence}\hfill
\begin{enumerate}\renewcommand{\theenumi}{\roman{enumi}}
\item There are at least $n-1$ KR-equivalence classes of separable potentials of degree $n$.
\item Gornik potentials form an equivalence class, and for $n > 2$, it is not generic.
\end{enumerate}
\end{theorem}

We will prove this theorem by analyzing the equivariant $\sl_n$-cohomology of the trefoil.
It is a good exercise to compute it for general $n$ using \Cref{thm:krasner};
but let us follow a different route here, which treats the cohomology theories more as a black box.
The proof is split in several lemmas, and uses the following theorem, which was proved quite recently:

\begin{theorem}[\cite{folklore}]\label{thm:folklore}
Let $\partial w$ be a potential with distinct roots $\alpha_1, \ldots, \alpha_k$.
For every link $L$, there is an isomorphism respecting the homological degree (but not the quantum degree)
\[
H_{\partial w}(L) \equiv \bigoplus_{i=1}^k H_{x^{(\mult_{\alpha_i}\partial w)}}(L).
\]
\end{theorem}

\noindent Here we are writing $\mult_{\alpha}f$ to mean the multiplicity of the root $\alpha$ in the complex polynomial $f$.

\begin{lemma}\label{lem:jordan}
	Let $\partial w$ be a potential with distinct roots $\alpha_1, \ldots, \alpha_k$ and
	let $d\in \mathbb{C}[x]$. Denote the $\mathbb{C}$-endomorphism
	of $\mathbb{C}[x]/\partial w$ given by multiplication with $d$ by $M_d$. Then
	\[
	\dim_{\mathbb{C}} \ker M_d = \sum_{i=1}^k \min\{\mult_{\alpha_i}\partial w, \mult_{\alpha_i}d\}.
	\]
\end{lemma}
\begin{proof}
	This follows from the decomposition as a $\C[x]$-module
	
	\[ \C[x]/\partial w = \bigoplus_{i=1}^k \C[x]/(x-\alpha_i)^{\mult_{\alpha_i}\partial w} {\rm .} \]
\end{proof}
\begin{lemma}\label{lem:lin}
Let
\[
q = x^{n-1} + \sum_{i=0}^{n-2} b_i x^i \in R_n[x], \quad b_i\in R_n
\]
be homogeneous of degree $2n-2$ with the following property:
for all $e: R_n\to \C$, let $\alpha_1, \ldots, \alpha_k$ be the roots of $e^*(p)$, then
\begin{equation}\tag{$\dagger$}
n - k  = 
\sum_{i=1}^{k} \min\{\mult_{\alpha_i}e^*(p), \mult_{\alpha_i}e^*(q)\}.
\end{equation}
Then the coefficient of $a_{i+1}$ as monomial in $b_i$ is not equal to $1$.
\end{lemma}
\begin{proof}
Let $r \in \{1,\ldots,n-1\}$. Assume the statement were not true for $b_{r-1}$.
This implies in particular, that for $e: R_n\to\C$ given by $e^*(p) = x^n + x^r = x^r(x^{n-r} + 1)$,
we have $e(b_{r-1}) = 1$. We will show that this leads to a contradiction.

The polynomial $e^*(p)$ has $k = n - r + 1$ different roots, and so the left-hand side of ($\dagger$)
equals $r - 1$.
Every root but $0$ has multiplicity $1$ and thus contributes at most $1$ to the sum on the right-hand side.
So the right-hand side is less or equal than $n - r + \min\{\mult_0e^*(p), \mult_0e^*(q)\}$, and thus
\begin{align*}
r - 1      & \leq n - r + \min\{n-r, \mult_0e^*(q)\} \\
\Rightarrow\ 2r - n - 1 & \leq \min\{n-r, \mult_0e^*(q)\} \\
\Rightarrow\ 2r - n - 1 & \leq \mult_0e^*(q).
\end{align*}
Hence we have $e(b_i) = 0$ for all $i \leq 2r - n - 2$.
For degree reasons,
\begin{align*}
e(b_i) \neq 0 & \implies \deg a_r | \deg b_i \\
              & \implies (2n - 2r) | (2n - 2i - 2).
\end{align*}
So if $2r - n - 1 < 0$, then $b_{r-1}$ is the only $b_i$ with non-zero evaluation, and we have
$e^*(q) = x^{r-1}(x^{n-r} + 1)$, contradicting ($\dagger$).

If, on the other hand, $2r - n - 1 \geq 0$, then $e(b_{2r-n-1}) \neq 0$ is possible. In that case
\begin{align*}
e^*(q) & = x^{n-1} + x^{r-1} + e(b_{2r - n - 1})x^{2r - n - 1} \\
& = x^{2r-n - 1}(x^{2n - 2r} + x^{n - r} + e(b_{2r - n - 1})).
\end{align*}
We have $\mult_0 e^*(q) = 2r - n - 1$, which in turn implies
that all other roots of $e^*(q)$ must be common roots with $e^*(p)$. 
Hence $x^{n-r} + 1$ divides $x^{2n - 2r} + x^{n - r} + e(b_{2r - n - 1})$, which contradicts $e(b_{2r - n - 1}) \neq 0$.
\end{proof}
\begin{remark}
One may be tempted to think that the hypotheses of the previous lemma are in fact
sufficient to show that
\[
q = \frac{1}{n}\cdot \frac{\partial p}{\partial x}.
\]
But this is not true, and indeed for $n = 3$ we have the following counterexample:
\[
q = \frac{x^3 + 2a_2x^2 + (4a_1 - a_2^2)}{3}.
\]
\end{remark}
\begin{lemma}\label{prop:trefoil}
Let $\partial w = x^n + a_{n-2}x^{n-2} + \ldots + a_0 \in \C[x]$ with $\partial w \neq x^n$,
and let $E$ be the spectral sequence associated to $C_{\partial w}(K)$.
Let $\ell$ be the smallest positive number such that $a_{n - \ell} \neq 0$ (i.e. $2 \leq \ell \leq n$).
\begin{enumerate}\renewcommand{\theenumi}{\roman{enumi}}
\item The pages $E_2, \ldots, E_{\ell}$ are insignificant.
\item For $K = T_{2,3}$, the page $E_{1 + \ell}$ is significant.
\end{enumerate}
\end{lemma}
\begin{proof}
For any knot $K$ with a diagram $D$, $C_{\eqva}(D)$ is by Gauss elimination homotopy equivalent to
a complex of free modules whose differentials have matrices 
all of whose entries are non-units in $R_n[x]/p$.
We shall assume that we have performed such a Gauss elimination and we shall abuse notation and write the new complex as $C_{\eqva}(K)$. 

Forgetting the action of $x$ gives a chain complex $\overline{C}_{\eqva}(D)$ of free $R_n$-modules.
One may continue Gaussian elimination as long as possible, arriving at a complex $\overline{C}_{\eqva}(K)$.
Evaluating this chain complex by some $e: R_n \to \C$ gives a chain complex homotopy equivalent to $C_{e^*(p)}$.
All non-vanishing matrix entries in $\overline{C}_{\eqva}(K)$ are homogeneous non-constant polynomials in $R_n$;
so the degree of such an entry is at least $2\ell$.
This implies part (i).

To obtain reduced $\sl_n$-cohomology from $C_{\eqva}(K)$, one may evaluate by the map %
that sends all $a_i$ and $x$ to $0$. So $C_{\eqva}(K)$ has the same Poincar\'{e}-polynomial
as $\widetilde{H}_{x^n}(K)$.
The reduced Homflypt-cohomology
of the trefoil has Poincar\'{e}-polynomial
\[
a^2q^{-2} + t^2a^2q^2 + t^3a^4.
\]
This can be easily computed from the Homflypt-polynomial and the signature, since the trefoil is a two-bridge knot and thus KR-thin.
That also implies that Rasmussen's spectral sequences are all trivial, and hence the reduced $\sl_n$-cohomology
is obtained from the Homflypt-cohomology simply by the regrading $a\mapsto q^n$. It has therefore Poincar\'{e}-polynomial
\[
q^{2n-2} + t^2q^{2n+2} + t^3q^{4n}.
\]
Next, the differential between homological degree $2$ and $3$ is given by multiplication with a polynomial $d\in R_n[x]$
which is homogeneous of degree $2n-2$.
Using \Cref{thm:krasner}, one could compute by hand that $d = \partial p / \partial x$.
Instead, we proceed as follows: let $e:R_n\to\C$ send all $a_i$ to $0$. Applying $e$ to $C_{\eqva}(K)$ and forgetting the action of $x$
will give give unreduced $\sl_n$-cohomology. If $e(d)$ were $0$, then we would have $H_{x^n}(K) = H_{x^n}(U) \otimes \widetilde{H}_{x^n}(K)$,
where $U$ is the unknot. But this is impossible since there is a spectral sequence induced by $C_{x^n-1}(K)$
from $H_{x^n}(K)$ which respects the quantum
degree modulo $2n$ and whose limit is supported in cohomological degree $0$.
Therefore $e(d)\neq 0$, and hence $e(d)$ is a non-zero scalar multiple of $x^{n-1}$.
This gives $\dim_{\C} H^2_{x^n}(K) = n - 1$, and so \Cref{thm:folklore} implies that for all $e:R_n\to\C$,
where $e^*(p)$ has distinct roots $\alpha_1, \ldots, \alpha_k$, we have
\[
\dim_{\C} H^2_{e^*(p)}(K) = n - k.
\]
On the other hand, \Cref{lem:jordan} implies that
\[
\dim_{\C} H^2_{e^*(p)}(K) = \sum_{i=1}^k \min\{\mult_{\alpha_i}\partial w, \mult_{\alpha_i}d\}.
\]
So the hypotheses of \Cref{lem:lin} are satisfied by $d$.

Now let us examine what happens when we pass to unreduced cohomology: this simply means forgetting the action of $x$,
thus obtaining a cochain complex of vector spaces. With respect to the basis $(1, x, \ldots, x^{n-1})$ 
of $\mathbb{C}[x]/\partial w$, the differential between homological degree $2$ and $3$ is an $n\times n$-matrix $M$,
whose $(i,j)$-th entry is the coefficient of $x^{i-1}$ of the unique polynomial of degree at most $n-1$ that equals $x^{j-1}\cdot d$ in $\mathbb{C}[x]/\partial w$.
So the first two columns of $M$ can be computed as (recall that w.l.o.g. we set $a_{n-1}=0$)
\[
\begin{pmatrix}
b_0          & -a_0            &      &  \\
b_1          & b_0 - a_1       &      &  \\
\vdots       & \vdots          &      & \cdots \\
b_{n-3}      & b_{n-4}-a_{n-3} &      &  \\
0            & b_{n-3}-a_{n-2} &      &  \\
1            & 0               &      &
\end{pmatrix}
\]
Applying Gauss elimination to the entry $n$ at $(n,1)$ gives an $(n-1)\times(n-1)$ matrix whose first column is
\[
\begin{pmatrix}
-a_0             &      &  \\
b_0 - a_1        &      &  \\
\vdots           &      & \cdots \\
b_{n-4}-a_{n-3}  &      &  \\
b_{n-3}-a_{n-2}  &      &  \\
\end{pmatrix}
\]
We have already argued that all pages $E_2, \ldots, E_{\ell}$ of the spectral sequence are insignificant for degree reasons.
Now because of \Cref{lem:lin}, the differential on $E_{\ell}$ is non-trivial, and so $E_{1 + \ell}$ is significant.
\end{proof}
\begin{proof}[Proof of {\protect\Cref{cor:existence}}]
For (i), one can take e.g. $x^n + 1$ and  $x^n + x^i + a_0$ for $i\in\{1, \ldots, n-2\}$ and some $a_0 \in\mathbb{C}$ such that the polynomial is separable. By \Cref{prop:trefoil}, the cohomology of the trefoil associated to these potentials have spectral sequences with different significant pages, and are thus pairwise not KR-equivalent.

The second part (ii) follows from the fact that every potential is by \Cref{prop:linear_equivalence} KR-equivalent to
one with $a_{n-1} = 0$. To be KR-equivalent to a Gornik potential, the next significant page of the spectral sequence
needs to be the $(n+1)$-st, and this can only be the case if $a_i = 0$ for all $i > 0$.
\end{proof}
\begin{table}[h]
\begin{tabular}{c|c|c}
 & \multicolumn{2}{c}{\# KR-equivalence classes of separable potentials\ldots} \\[1ex]
$\deg \partial w$ & \ldots of the trefoil & \ldots of $P(5,-3,2)^{\#2}$  \\\hline
$2$ &  1           &  1           \\
$3$ &  2           &  $\geq 3$    \\
$4$ &  4           &  $\geq 6$    \\
$5$ &  8           &  $\geq 10$  \\\hline
$n$ &  $(n-1)! \geq \ldots \geq n-1$  &  $\geq n-1$ \\
\end{tabular}
\vspace{5mm}
\caption{Number of KR-equivalence classes of separable potentials for the trefoil and a more complicated knot;
for the former, classes can be determined precisely by continuing the calculation done in the proof of \Cref{prop:trefoil}.
For the latter, we calculated cohomologies with respect to a large batch of potentials with small coefficients.}
\label{table:KR}
\end{table}
\begin{remark}
We picked the trefoil for ease of calculation, and to demonstrate that even over the simplest non-trivial
knot there are at least $n-1$ different KR-equivalence classes. In fact, the numbers in \Cref{table:KR}
and a close look at the calculations suggest
that the actual number of classes might rather be $2^{n-2}$.

Note that there are non-KR-equivalent potentials that are KR-equivalent over the trefoil:
for example, $x^3-x\sim_{T_{3,2}} x^3 - x - 1$, but $x^3-x\not\sim_{P(5,-3-2)} x^3 - x - 1$.
Hence the differentials of $C_{\eqva}(P(5,-3,-2))$ are not all equal to $\partial w'$.
It would certainly be worthwhile to analyze which forms $C_{\eqva}$ takes in general, or
for certain classes of knots: for example, it could be the case that the equivariant cohomology
of two-bridge knots decomposes as sum of $\mathbb{C}[x]/\partial w$ (in cohomological degree 0) and
several summands of the form
\[
\mathbb{C}[x]/\partial w \xrightarrow{\partial w'} \mathbb{C}[x]/\partial w.
\]

Also, all separable potentials yield the same $E_{\infty}$ page for the trefoil; but it seems
a reasonable conjecture that there are knots (sufficiently complicated and certainly not positive)
for which the different KR-equivalence classes actually yield different $E_{\infty}$-pages.
\end{remark}

\section{Further illuminating examples}
\label{subsec:further_egs}
We have already seen through the example of \Cref{subsec:an_example} that the behavior of Khovanov-Rozansky with a separable potential can be quite unexpected, especially if one's intuition comes from Lee homology and Rasmussen's invariant.  However, the structural results that we have proven in \Cref{sec:lowerbounds} constrain this behavior to some extent.  There are some natural questions concerned with how unruly the invariants can be, and whether one might expect to be able to give much stronger constraints than we have hitherto done.

In this \namecref{subsec:further_egs} we list some of these natural questions and indicate through (computational) examples where the answer lies.

\begin{question}
Are there knots whose sliceness is not obstructed by any of the reduced concordance homomorphisms,
but is obstructed by some of the unreduced concordance invariants?
\end{question}
Let $K = P(9,-7,6) \# P(-7,5,-4)$. Then for all $n\geq 2$, we have $s_n(K) = 0$ \cite{lew2}, so none
of the generalized Rasmussen concordance homomorphisms obstruct the sliceness of $K$.
Neither do any of the reduced concordance homomorphisms we checked.
However, \khoca{} calculates the Poincar\'{e}-polynomial of $H_{x^3-x}(K)$ as $2 + q^2$, which shows that $K$ is not slice.

\begin{question}
In \Cref{gornikniceness}, we have shown that all roots $\alpha$ of a Gornik potential $\partial w$ give the same reduced
concordance homomorphism $\widetilde{s}_{\partial w,\alpha}$. More generally, the symmetry of potentials such
as $x^3 - x$, which is projectively invariant under $x\mapsto -x$, extends to their reduced concordance homomorphisms:
we have $\widetilde{s}_{x^3-x,1} = \widetilde{s}_{x^3-x,-1}$ by \Cref{reducedsymmetry}.
Is it actually true for every potential that all roots give the same reduced concordance homomorphism?
\end{question}
No, for example, we have that $\widetilde{s}_{x^5 - x, 0}(P(5,-3,2)) = 0$, but $\widetilde{s}_{x^5 - x, \alpha}(P(5,-3,2)) = -1/4$
for $\alpha \in \{\pm 1,\pm i\}$, as can be computed with \khoca.

\begin{question}\label{q:shapes}
We have seen that unreduced cohomology does not always have a Poincar\'{e} polynomial of the form $q^{2(n-1)s}\cdot[n]$ with $s\in\mathbb{Z}$.
What shapes does it take? For example, are the generators always in quantum degrees close to each other?
\end{question}
For $K_i = P(5,-3,2)^{\# i}$, we have for all $i \geq 0: \widetilde{s}_{x^5-x,0}(K_i) = 0$, but $\widetilde{s}_{x^5-x,\alpha}(K_i) = -i/4$:
as $i$ grows, so does the distance between the reduced concordance homomorphism of the root $0$, and the other four.
Since the distance between the $\widetilde{s}$ and the unreduced $j_i$ is bounded above by \Cref{prop:relation red unred},
the shape of unreduced $(x^5-x)$-cohomology of $K_i$ is increasingly elongated with growing $i$. And indeed, \khoca{} calculations
for $i = 1,2,3,4$ suggest that its Poincar\'{e} polynomial is $q^{-4} + q^{-1-2i}\cdot[4]$ for all $i\geq 1$.

\begin{question}
	\label{q:quasihom}
\Cref{prop:quasihomeg} shows how to get a quasi-homomorphism from the smooth concordance group to the rationals using unreduced cohomology.
For $H_{x^n - 1}$, this is a homomorphism. Is there any way to define a homomorphism for other potentials?
\end{question}
There may be, but if we take for example $\partial w = x^3 + x + 1$, then it cannot be done in an obvious way, as the following proposition indicates:
\begin{proposition}
Let $j_i(K)$ be defined as in the introduction with potential $x^3 + x + 1$.
Suppose $\varphi$ is a function from the set $\{(x_1, x_2, x_3) \mid x_1 \leq x_2 \leq x_3\} \subset (2\mathbb{Z})^3$ to $\mathbb{R}$,
such that $\varphi^*: K \mapsto \varphi(j_1(K), j_2(K), j_3(K))$ is a concordance homomorphism. Then $\varphi^*$ takes all knots $K$ with
$H_{\partial w}(K) =  q^{2(n-1)s}\cdot[n]$ to zero.
\end{proposition}
\begin{proof}
Let $K = P(7,-5,4)^{\# 2}$. By \khoca-calculations, we have
\[
\forall i\in\{1,2,3\}: j_i(K) = j_i(K\# K) = -2.
\]
Therefore, $\varphi^*(K\# K) = \varphi^*(K) \implies \varphi^*(K) = \varphi(-2,-2,-2) = 0$. By taking the mirror image, we get $\varphi(2,2,2) = 0$ as well.
But we have $\varphi^*(K \# T_{-2,3}) = \varphi(2,2,2) = 0 \implies \varphi^*(T_{-2,3}) = 0 \implies \varphi(2,4,6) = 0$. Therefore $\varphi^*$ sends also
any multiple of $T_{-2,3}$ to zero, and we have $\forall s\in\mathbb{Z}: \varphi(2 + 4s, 4 + 4s, 6 + 6s) = 0$.
\end{proof}

This implies that if one can define such a homomorphism then it must be identically $0$ on all quasi-positive and homogeneous knots, which would be very unusual behavior indeed.  In fact, based on wider calculations of $H_{x^3 + x + 1}$ which we do not report here, it seems very likely that any such homomorphism defined as in the proposition will be identically $0$.

\begin{question}
We have seen the effect on unreduced cohomology $H_{\partial w}(K)$ of taking the connected sum with homogeneous and quasi-positive knots $K'$ in \Cref{thm:homogeneousetc}:
the cohomology $H_{\partial w}(K \# K')$ is just a quantum shift of $H_{\partial w}(K)$.  But perhaps it is not the quasi-positivity or homogeneity of $K'$ that is important, but just the shape of the associated graded vector space to its cohomology (which is that of a shifted unknot).
Is the result more generally true for knots $K'$ with $H_{\partial w}(K') = q^{2(n-1)s}\cdot[n]$, i.e. does $H_{\partial w}(K \# K') = H_{\partial w}(K)[2(n-1)s]$ hold?
\end{question}
As discussed in the answer to \Cref{q:shapes}, we have
\[
H_{x^5-x}(P(5,-3,2)^{\#4}) = q^{-4} + q^{-9}\cdot [4] = q^{-8}\cdot [5],
\]
but e.g.
\[
H_{x^5-x}(P(5,-3,2)^{\#4} \# P(5,-3,2)^{\#4}) = q^{-4} + q^{-17}\cdot [4] \neq q^{-16}\cdot [5].
\]

\begin{question}\label{q:sumdet}
Is $H_{\partial w}(K\# K')$ determined by $H_{\partial w}(K)$ and $H_{\partial w}(K')$?
\end{question}
No -- take $K, K' \in \{T(3,2), P(5,-3,2)^{\#4}\}$ (see the previous Question).

\begin{question}
Is it possible that the reduced concordance homomorphisms arising from degree $3$ polynomials are all just linear combinations of $s_2$ and $s_3$?\footnote{We thank Mikhail Khovanov for raising this question.}
\end{question}
No.  We consider the reduced concordance invariant given by taking the root $x=0$ of the potential $x^3 - x$.  Then computing the invariants for the trefoil knot and the knot $P(5,-3,2)$ one can deduce that if there is such a linear dependence it is of the form:
\[
\reds_{x^3-x,0} = 2s_3 - s_2 {\rm .}
\]
Next, consider the pretzel knot $P(7,-5,4)$.  In \cite{lew2} the first author showed that this knot satisfies $s_2(P(7,-5,4)) = 1$ and $s_n(P(7,-5,4)) = 0$ for any $n>2$.  We can compute the reduced $\sl(3)$ cohomology (using for example \cite{lew1}) and see that in cohomological degree $0$ the cohomology is supported in quantum degree $0$.  Hence in particular $\reds_{x^3-x,0}(P(7,-5,4)) = 0$.  This then shows that $\reds_{x^3-x,0}$ is not in the span of $s_2$ and $s_3$.

\section{Computer calculations}
\label{sec:calc}
\subsection{Bipartite links}
\label{sec:bipartite}
In this section, we consider oriented links with \emph{matched diagrams}, that is to say,
diagrams obtained by gluing together copies of the basic 2-crossing tangle (and its mirror-image) as shown in \Cref{fig:matched}a.
\begin{figure}[t]
(a) \includegraphics{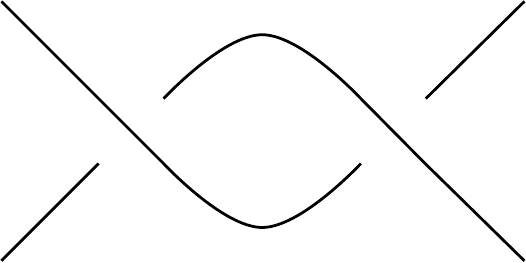}\qquad
(b) \includegraphics{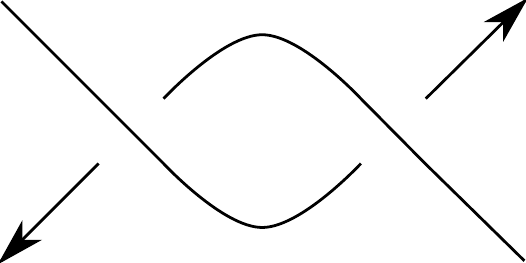}\qquad
\caption{The basic matched tangle, (a) unoriented and (b) oriented.}
\label{fig:matched}
\end{figure}
Such links are called \emph{bipartite links}. If the orientations of the tangles are always as in \Cref{fig:matched}b (or its mirror image),
we call the diagram \emph{orientedly matched} and the link \emph{orientedly bipartite}.
\begin{prop}\label{lem:bip2}
An unoriented matched link diagram $D$ admits an orientation that makes it orientedly matched.
This orientation is unique up to overall reversals of orientations of disjoint diagram components of $D$.
\end{prop}
\begin{proof}
If $D$ is a knot, this is asserted without proof in \cite{bipartite}; and indeed, pick one of the basic tangles: then the two strands in
the complement of the tangle pair up its four endpoints. A priori there are three different pairings possible; but pairing the upper two
endpoints would give a two-component link, and pairing each endpoint with the one diametrically opposed would imply that the complement
of the tangle has an odd number of crossings. So the left endpoints are paired, which implies that the tangle is oriented in the matched sense.

Assume now that $D$ has more than one component and is not split.
Then one can rotate a subset $S$ of the basic tangles constituting $D$ by a quarter-turn,
such that the result is a knot diagram $D'$.
Note that the set of orientations of $D$ that maked $D$ orientedly matched
are in 1-1 correspondence with the orientations of $D'$ that make $D'$ orientedly matched:
the correspondence is given by rotating each tangle in $S$ by a quarter-turn and reversing
its orientation.

If $D$ is split, treat every component separately. 
\end{proof}

Matched diagrams were introduced in \cite{przytycka}
in the context of the Homflypt-polynomial. The authors conjectured that there were non-bipartite knots, a problem
which remained open for 24 years, until it was solved by Duzhin and Shkolnikov \cite{bipartite},
who showed that if a higher Alexander ideal of a bipartite knot contains the polynomial $1+t$, then this ideal must be
trivial. Thus various of 9- and 10-crossing knots are shown to be not bipartite, among them
the $P(3,3,3)$-pretzel knot. In fact, this generalizes to $P(p, q, r)$-pretzel knots
with $p, q, r$ odd and $\lambda = \gcd(p,q,r) > 1$, because their second Alexander ideal is generated by $\lambda$ and $1+t$.
If, on the other hand, $p$ is even, then the $P(p,q,r)$-pretzel knot is bipartite, as we shall prove later on.
\begin{figure}[t]
\includegraphics{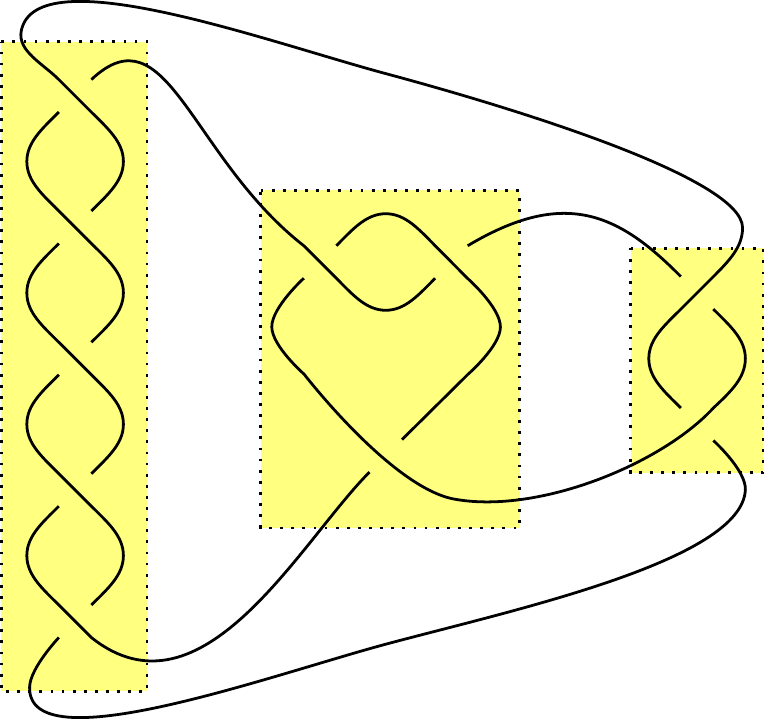}
\caption{The Montesinos-knot $(1/5,2/3,-1/2)$ (which is equivalent to $M(1/5,-1/3,1/2)$, the $(5,-3,2)$-pretzel knot).}
\label{fig:montesinos}
\end{figure}

Our interest in bipartite links is motivated by Krasner's discovery \cite{krasner} that the Khovanov-Rozansky cochain complexes
of the basic oriented matched tangle (\Cref{fig:matched}b), and consequently of orientedly matched diagrams take a particularly simple form:
they are homotopy equivalent to cochain complexes in the TQFT-subcategory -- avoiding MOY-graphs and foams.
In \Cref{thm:krasner}, Krasner's theorem is generalized to equivariant Khovanov-Rozansky cohomologies.

This observation has allowed us to write a computer program called \khoca{} that computes Khovanov-Rozansky cohomologies
of bipartite links. A description is given in the next section.

Duzhin and Shkolnikov prove that rational knots are bipartite; the following is a generalization,
rendering precise a remark of Przytycki's that `half of Montesinos knots should be bipartite'.
A \emph{Montesinos link} is a generalization of pretzel links, where the strands are replaced by \emph{rational tangles} --
see \Cref{fig:montesinos} for an example. Rational tangles up to boundary-fixing isotopy are in one-to-one correspondence with $\mathbb{Q}\cup\{\infty\}$ \cite{conway}.
The rational tangle with twists $a_1, \ldots, a_n$ corresponds to the value of the continued fraction
\[
[a_1, \ldots, a_k] := a_k + \frac{1}{a_{k-1} + \frac{1}{\ldots + a_1}}.
\]
So a Montesinos link may be written as $L = M(p_1/q_1, \ldots, p_n/q_n)$, where $p_i\in\mathbb{Z}, q_i \in \mathbb{Z}^+, (p_i, q_i) = 1$.
In this notation, e.g. $M(3)$ is the trefoil and $M(1/a, 1/b, 1/c)$ the $(a,b,c)$-pretzel link.
Clearly, without changing the isotopy type of $L$ one may insert a $0$ to or remove it from the list of fractions; and
\[
M(p_1/q_1, \ldots, p_n/q_n) = M(\mp 1, p_1/q_1, \ldots, p_i/q_i \pm 1, \ldots, p_n/q_n).
\]
\begin{theorem}\label{thm:montesinos}
Consider the unoriented Montesinos link $L=M(p_1/q_1, \ldots, p_n/q_n)$. %
If $L$ has more than one component, then it is bipartite. %
If $L$ is a knot and one of the denominators $q_i$ is even, then $L$ is bipartite.
\end{theorem}
\begin{lemma}[{\cite[Lemma 2]{duzhin}}] \label{lem:bip1}
If either $p$ or $q$ is even, then $p/q$ can be written as continued fraction
$p/q = [a_1, \ldots, a_k]$
with all $a_i$ even.
\end{lemma}
\begin{proof}[Proof of {\Cref{thm:montesinos}}]
Let $A = \{i \mid p_i, q_i \text{ odd}\}$.
The Montesinos link $L$ is isotopic to 
\[
M(\#A/1, p_1'/q_1, \ldots, p_n'/q_n),
\]
where $p_i' = p_i - q_i$ if $i\in A$ an $p_i'=p_i$ otherwise.
If $L$ has more than one component, and none of the $q_i$ is even, it follows that $\#A$ is even.
If, on the other hand, one of the $q_i$ is even, w.l.o.g. $q_1$, $L'$ is isotopic to
\[
M( (p_1+q_1\cdot \#A)/q_1, p_2'/q_2, \ldots, p_n'/q_n).
\]
So if one of the hypotheses of the \namecref{thm:montesinos} is satisfied, $L$ is isotopic to a Montesinos link who only contains rational tangles whose fractions have even numerator or denominator.
But by \Cref{lem:bip1}, the corresponding rational tangles correspond to a continued fraction $[a_1, \ldots, a_k]$ with all $a_i$ even, and can thus clearly be glued from copies of the basic unoriented matched tangle.
\end{proof}
\begin{lemma}[\cite{krasnerEquivariant}] \label{lem:moy}
In the category of equivariant matrix factorizations
the maps (i), (ii) and (iii) are filtered isomorphisms.
\begin{equation}\LeftEqNo\tag{i}
\bigoplus_{i=0}^{n-1} q^{2i+1-n} \varnothing \xrightarrow{\left(\begin{smallmatrix} c_0 & \cdots & c_{n-1} \end{smallmatrix}\right)} \avcfig{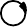},
\end{equation}
where $c_i$ is given by the composition of the following maps:
\[
q^{2i+1-n}\varnothing \xfixed{\iota}
q^{2i}\avcfig{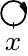} \xfixed{x^i}
\avcfig{circle}.
\]
\begin{equation}\LeftEqNo\tag{ii}
\avcfig{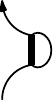} \xrightarrow{\left(\begin{smallmatrix} d_0 \\ \vdots \\ d_{n-2}\end{smallmatrix}\right)}
\bigoplus_{i=0}^{n-2} q^{2i+2-n} \avcfig{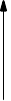}
\end{equation}
where $d_i$ is given by the composition of the following maps:
\[
\avcfig{digon} \xfixed{\chi_1}
q^{-1} \avcfig{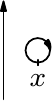} \xfixed{x^{n-2-i}}
q^{2i+3-2n} \avcfig{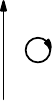} \xfixed{\varepsilon}
q^{2i+2-n} \avcfig{I}.
\]
\begin{equation}\LeftEqNo\tag{iii}
\avcfig{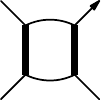} \xrightarrow{\left(\begin{smallmatrix} s_0 \\ \vdots \\ s_{n-3} \\ * \end{smallmatrix}\right)}
\bigoplus_{i=0}^{n-3} q^{2i+3-n} \avcfig{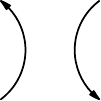} \oplus \avcfig{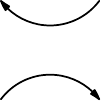}
\end{equation}
where $s_i$ is given by the composition of the following maps:
\[
\avcfig{square} \xfixed{\chi_1\chi'_1}
q^{-2} \avcfig{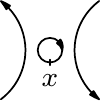} \xfixed{x^{n-3-i}}
q^{2i+4-2n} \avcfig{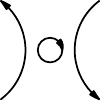} \xfixed{\varepsilon}
q^{2i+3-n} \avcfig{II}.
\]
\end{lemma}
\begin{theorem}\label{thm:krasner}
The following filtered cochain complexes are homotopy equivalent:
\begin{multline*}
C_{\eqva}\Biggl(\,
\avcfig{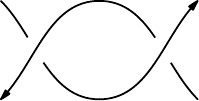}\,
\Biggr)
\ \cong\\
\underline{q^{1-n} \avcfig{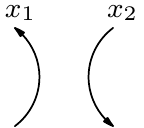}} \xfixed{x_1 - x_2}
q^{-1-n}\avcfig{II} \xfixed{\text{\emph{saddle}}}
q^{-2n}\avcfig{equals}.
\end{multline*}
\end{theorem}
\begin{proof}
We will compose a series of cochain homotopy equivalences to connect the two terms.
By definition,
\begin{multline*}
C_{\eqva}\Biggl(\,
\avcfig{krasner}\,
\Biggr)
\ = \\
\underline{q^{2-2n} \avcfig{IICircle}}\ \xfixed{\left(\begin{smallmatrix}d_1\\ *\end{smallmatrix}\right)}\ 
\begin{matrix}
q^{1-2n}\avcfig{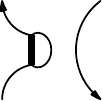}\\[2ex] \oplus \\[1ex]
q^{1-2n}\avcfig{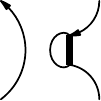}
\end{matrix}
\ \xfixed{\left(\begin{smallmatrix}* & e_1\end{smallmatrix}\right)}\ 
q^{-2n}\avcfig{square}.
\end{multline*}
Here and later, a star ($*$) indicates a map that we do not need to know.
To start, replace the circle in cohomological degree 0 and the first digon in cohomological degree 1
using the respective MOY-decompositions.
This leads to $d_1$ being replaced by an $(n-1)\times n$-matrix $d_2$.
For $0 \leq i \leq n-1$ and $0 \leq j \leq n-2$, its $(i+1,j+1)$-entry is a map 
from $q^{3-3n+2j}\,\avcfig[scale=0.3]{II}$ to $q^{3-3n+2i}\,\avcfig[scale=0.3]{II}$ given by the following composition (dotted line):

\[\xymatrix{
q^{3-3n+2j} \avcfig{II}
\ar@{..>}[d]
\ar[r]^-{\iota} &
q^{2-2n+2j} \avcfig{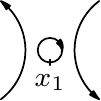}
\ar[r]^-{x_1^j} &
q^{2-2n} \avcfig{IICircle}
\ar[rd]^-{d_1} \\
q^{3-3n+2i} \avcfig{II} &
\ar[l]^-{\varepsilon} 
q^{4-4n+2i} \avcfig{IICircle} &
\ar[l]^-{x_2^{n-2-i}} 
q^{-2n} \avcfig{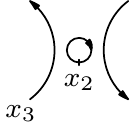} &
\ar[l]^-{\chi_1} 
q^{1-2n} \avcfig{digonLeft}
}\]

We have $\chi_1 d_1 = (x_3 - x_2)$, and thus the whole map equals $\varepsilon x_1^{n-2-i+j}(x_3-x_1)\iota$.
This is clearly equal to a multiple of the identity of $\avcfig[scale=0.3]{II}$ if $i = j$.
On the other hand, if $i < j$, then this map is zero. Thus $d_2$ is an upper triangular matrix whose main diagonal
consists of isomorphisms. Therefore the submatrix obtained by deleting the last column is invertible. So using Gauss elimination,
the cochain complex is homotopy equivalent to 
\[
\underline{q^{1-n} \avcfig{IICircle}}\ \xfixed{*}
q^{1-2n}\avcfig{digonRight}
\ \xfixed{e_1}\ 
q^{-2n}\avcfig{square}.
\]

To proceed, use MOY-decompositions again, to replace the remaining digon and the square.
For the digon, we will use the dual of the map given in \Cref{lem:moy} (ii).
In this way, $e_1$ is replaced by a $(n-1)\times (n-1)$-matrix $e_2$.
Let us ignore the last row and last column, and denote by $e_2'$ the corresponding submatrix.
For $i, j \in \{0,\ldots n-2\}$, its $(i+1,j+1)$-entry
is a map from $q^{3-3n+2j} \avcfig[scale=0.3]{II}$ to $q^{3-3n+2i} \avcfig[scale=0.3]{II}$, given by the following composition (dotted line):
\[\scalebox{0.98}{\xymatrix{
q^{3-3n+2j} \avcfig{II}
\ar@{..>}[d]
\ar[r]^-{\iota} &
q^{2-2n+2j} \avcfig{IICirclex1}
\ar[r]^-{x_1^j} &
q^{2-2n} \avcfig{IICircle}
\ar[r]^-{\chi_0} &
q^{1-2n} \avcfig{digonRight}
\ar[d]^-{e_1}\\
q^{3-3n+2i} \avcfig{II} &
\ar[l]^-{\varepsilon}
q^{4-4n+2i} \avcfig{IICircle} &
\ar[l]^-{x_2^{n-3-i}} 
q^{-2-2n} \avcfig{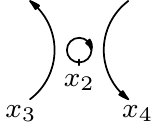} &
\ar[l]^-{\chi_1\chi_1'} 
q^{-2n} \avcfig{square}
}}\]

Because $\chi_1\chi_1'e_1\chi_0 = (x_3 - x_2)(x_2 - x_4)$, the whole map equals $\varepsilon x_1^{n-3-i+j}(x_3 - x_1)(x_1 - x_4) \iota$.
As before, this is a non-zero multiple of the identity for $i = j$, and vanishes for $i < j$.
Hence $e_2'$ is an invertible submatrix of $e_2$. By Gauss elimination, our cochain complex is homotopy equivalent to
\[
\underline{q^{1-n} \avcfig{IIxx}} \xfixed{d_3}
q^{-1-n}\avcfig{II} \xfixed{e_3}
q^{-2n}\avcfig{equals}.
\]

To determine the maps, note that the Hom-space $(\avcfig[scale=0.3]{II}, \avcfig[scale=0.3]{equals})$
is one-dimensional in the $q$-degree in question. Thus $e_3$ is a multiple of the saddle.
Now, close off the original tangle to the unknot. The cohomology of
the unknot has support in cohomological degree $0$; but if $e_3$ were $0$, the above complex
would have cohomology in cohomological degree $2$ after closing off. Hence $e_3$ is a non-zero multiple
of the saddle.

The Hom-space $(\avcfig[scale=0.3]{II}, \avcfig[scale=0.3]{II})$ of the $q$-degree
in question is two-dimensional, but only the subspace generated by $(x_1 - x_2)$ yields $0$ when
composed with the saddle. Closing off as before, we see that the dimension of the first cochain group
is strictly greater than the dimension of the second. So, for the first cohomology group
to vanish, $d_3$ needs to be non-zero. Thus $d_3$ is a non-zero multiple of $(x_1 - x_2)$.

A final isomorphism of cochain complexes may be used to do away with the non-zero factors.
\end{proof}
\begin{remark}\label{rmk:z}
To our knowledge, there is no integral Khovanov-Rozansky cohomology theory yet
that is defined for arbitrary tangles (not just pieces of braids as 
in \cite{krasnerIntegral}).
But if such a theory can be defined based on the Khovanov-Rozansky cube of singular resolutions, it is likely to satisfy \Cref{thm:krasner}.

In particular, \Cref{thm:krasner} gives a complex associated to a matched diagram that is defined over the integers.  Hence one can compute a cohomology theory $H$ over the integers for matched diagrams and thus make conjectural computations of the as-yet-undefined integral Khovanov-Rozansky cohomology.
\end{remark}

\subsection{A computer program}
\label{subsec:compu}
\label{sec:examples}
Although there is a variety of computer programs doing computations
in Khovanov-Rozansky cohomologies, none of them can quickly calculate
$\mathfrak{sl}_n$-cohomologies for small $n$ of small
knots.
One reason is the difficulty of implementing the calculus
of MOY-graphs and matrix factorisations (or some other formalism
describing the differentials) on the tangle level, which is necessary
for Bar-Natan's divide-and-conquer algorithm \cite{fastcompu}.
\Cref{thm:krasner} shows that to compute the cohomology of \emph{bipartite}
knots, a computer program only needs to do calculations in the TQFT-category
of tangles and cobordisms. This category is much easier for a computer program.
While not all knots are bipartite, and regrettably most torus knots appear not to be,
there are still enough bipartite knots which are not two-bridge and have interesting Khovanov-Rozansky cohomologies,
notably the odd-odd-even pretzel knots, which are our main source of examples.

Our program \khoca{} calculates unreduced and reduced $\mathfrak{sl}_n$-cohomology (including all pages of the spectral sequence) of bipartite knots,
for arbitrary potentials of arbitrary degree, over the complex numbers, integers and finite prime fields
(beware: for $n \geq 4$, the results over integers and prime fields have not been proven correct, cf. \Cref{rmk:z}).
Thanks to the divide-and-conquer algorithm (and implementation details such as sparse matrices and multiprocessing)
it does so in reasonable time, e.g. the calculation of $P(11,-9,8)$ (a 28-crossing knot) over some random potential of degree $5$ over the integers takes
five minutes.
Some of the examples calculated with \khoca{} can be found in \Cref{subsec:bunch}.
The program will shortly be made publicly available \cite{khoca-pub}.

\section{Outlook}
Throughout the text, we have worked over the complex numbers.
However, we expect our results to generalize to
\textbf{cohomologies over finite fields}, yielding different slice genus lower bounds.

The \textbf{knight move conjecture} arose quickly after Khovanov cohomology \cite{garoufalidis,dror},
but is still open; phrased in the language of this article, it simply states that
the spectral sequence of $C_{x^2-1}$ (over the complex numbers) collapses on the third page (which is the first significant
page after $E_1$). There is some weak evidence against the conjecture: no `reason why it should be true' is known,
and the lack of a counterexample could simply come from our limited ability to calculate cohomology
of large knots. Moreover, generalizations of the conjecture fail: e.g., the spectral sequences of $C_{x^2-1}(T(7,8))$
over $\mathbb{F}_7$, $C_{x^2-1}(T(6,7))$ over $\mathbb{F}_3$ or $C_{x^3 - 1}(T(5,6))$ over $\mathbb{F}_5$
collapse only on the \emph{second} significant page after $E_1$ (\cite{fastcompu}, and calculations with \texttt{foamho} \cite{foamho}).
Nevertheless, it might be noteworthy that all small knots $K$ that we considered displayed the following behavior:
let $E$ be the spectral sequence of $C_{\partial w}(K)$ (over $\mathbb{C}$), then $E_{2\deg\partial w - 1} = E_{\infty}$.

There is a \textbf{new potential topological application} of the invariants:
we have seen that the sliceness obstructions arising from unreduced cohomologies are not all equivalent to concordance homomorphisms.
So they could potentially be used to prove the non-sliceness of a knot that represents torsion in the concordance group,
such as an amphichiral knot.  We do not know, for example, of a reason why for some amphichiral knot $K$ and
some separable potential $\partial w$, we could not have
\begin{align*}
H_{\partial w}(K) & = 3 & \text{where $\deg\partial w = 3$,} \\
\text{or}\quad H_{\partial w}(K) & = 2q^{-1} + 2q &\text{where $\deg\partial w = 4$}.
\end{align*}
Note that either of these is in accordance with \Cref{prop:relation red unred} and \Cref{prop:mirror},
and obstructs sliceness since $H_{\partial w}(U) = [\deg\partial w]$.
In contrast, invariants such as knot signatures or slice-torus invariants must necessarily vanish
on such knots.

\bibliographystyle{myamsalpha}
\bibliography{main}

\providecommand{\bysame}{\leavevmode\hbox to3em{\hrulefill}\thinspace}
\providecommand{\MR}{\relax\ifhmode\unskip\space\fi MR }
\providecommand{\MRhref}[2]{%
  \href{http://www.ams.org/mathscinet-getitem?mr=#1}{#2}
}
\providecommand{\href}[2]{#2}
\begin{thebibliography}{MTV07}

\bibitem[BN02]{dror}
Dror Bar-Natan, \emph{On {K}hovanov's categorification of the {J}ones
  polynomial}, Algebr. Geom. Topol. \textbf{2} (2002), 337 \xox{MR}{1917056}
  \xox{arXiv}{math/0201043}.

\bibitem[BN07]{fastcompu}
\bysame, \emph{Fast {K}hovanov homology computations}, J. Knot Theory
  Ramifications \textbf{16} (2007), no.~3, 243--255 \xox{MR}{2320156}
  \xox{arXiv}{math/0606318}.

\bibitem[Con70]{conway}
John~H. Conway, \emph{An enumeration of knots and links, and some of their
  algebraic properties}, Computational {P}roblems in {A}bstract {A}lgebra
  ({P}roc. {C}onf., {O}xford, 1967), Pergamon, Oxford, 1970, pp.~329--358
  \xox{MR}{0258014}.

\bibitem[DS10]{duzhin}
Sergei Duzhin and Mikhail Shkolnikov, \emph{{A formula for the HOMFLY
  polynomial of rational links}}, 2010 \xox{arXiv}{1009.1800} to appear in
  Arnold Math. J.

\bibitem[DS14]{bipartite}
\bysame, \emph{Bipartite knots}, Fund. Math. \textbf{225} (2014), 95--102
  \xox{MR}{3205567} \xox{arXiv}{1105.1264}.

\bibitem[Gar04]{garoufalidis}
Stavros Garoufalidis, \emph{A conjecture on {K}hovanov's invariants}, Fund.
  Math. \textbf{184} (2004), 99--101 \xox{MR}{2128045}.

\bibitem[Gor04]{G}
Bojan Gornik, \emph{Note on {K}hovanov link cohomology}, 2004
  \xox{arXiv}{math/0402266}.

\bibitem[HN13]{heddni}
Matthew Hedden and Yi~Ni, \emph{Khovanov module and the detection of unlinks},
  Geom. Topol. \textbf{17} (2013), no.~5, 3027--3076 \xox{MR}{3190305}
  \xox{arXiv}{1204.0960}.

\bibitem[KR08]{KR1}
Mikhail Khovanov and Lev Rozansky, \emph{Matrix factorizations and link
  homology}, Fund. Math. \textbf{199} (2008), no.~1, 1--91 \xox{MR}{2391017}
  \xox{arXiv}{math/0401268}.

\bibitem[Kra09]{krasner}
Daniel Krasner, \emph{A computation in {K}hovanov-{R}ozansky homology}, Fund.
  Math. \textbf{203} (2009), no.~1, 75--95 \xox{MR}{2491784}
  \xox{arXiv}{0801.4018}.

\bibitem[Kra10a]{krasnerEquivariant}
\bysame, \emph{Equivariant {${\rm sl}(n)$}-link homology}, Algebr. Geom. Topol.
  \textbf{10} (2010), no.~1, 1--32 \xox{MR}{2580427} \xox{arXiv}{0804.3751}.

\bibitem[Kra10b]{krasnerIntegral}
\bysame, \emph{Integral {HOMFLY}-{PT} and {${\rm sl}(n)$}-link homology}, Int.
  J. Math. Math. Sci. (2010), Art. ID 896879, 25 \xox{MR}{2726290}
  \xox{arXiv}{0910.1790}.

\bibitem[Lee05]{Lee}
Eun~Soo Lee, \emph{An endomorphism of the {K}hovanov invariant}, Adv. Math.
  \textbf{197} (2005), no.~2, 554--586 \xox{MR}{2173845}
  \xox{arXiv}{math/0210213}.

\bibitem[Lew13a]{foamho}
Lukas Lewark, \emph{Foam{H}o},
  2013\qua\url{http://lewark.de/lukas/foamho.html},  computer program.

\bibitem[Lew13b]{lew1}
\bysame, \emph{$\mathfrak{sl}_3$-foam homology calculations}, Algebr. Geom.
  Topol. \textbf{13} (2013), no.~6, 3661--3686 (electronic) \xox{MR}{3248745}
  \xox{arXiv}{1212.2553}.

\bibitem[Lew14]{lew2}
\bysame, \emph{Rasmussen's spectral sequences and the
  $\mathfrak{sl}_n$-concordance invariants}, Adv. Math. \textbf{260} (2014),
  59--83 \xox{MR}{3209349} \xox{arXiv}{1310.3100}.

\bibitem[Lew15]{khoca-pub}
\bysame, \emph{Phenomenology of {K}hovanov-{R}ozansky homologies}, 2015
  forthcoming paper.

\bibitem[Liv04]{livingston}
Charles Livingston, \emph{Computations of the {O}zsv\'ath-{S}zab\'o knot
  concordance invariant}, Geom. Topol. \textbf{8} (2004), 735--742 (electronic)
  \xox{MR}{2057779} \xox{arXiv}{math/0311036}.

\bibitem[Lob09]{Lobb1}
Andrew Lobb, \emph{A slice genus lower bound from {${\rm sl}(n)$}
  {K}hovanov-{R}ozansky homology}, Adv. Math. \textbf{222} (2009), no.~4,
  1220--1276 \xox{MR}{2554935} \xox{arXiv}{math/0702393}.

\bibitem[Lob11]{lobbineq}
\bysame, \emph{Computable bounds for {R}asmussen's concordance invariant},
  Compos. Math. \textbf{147} (2011), no.~2, 661--668 \xox{MR}{2776617}
  \xox{arXiv}{0908.2745}.

\bibitem[Lob12]{Lobb3}
\bysame, \emph{A note on {G}ornik's perturbation of {K}hovanov-{R}ozansky
  homology}, Algebr. Geom. Topol. \textbf{12} (2012), no.~1, 293--305
  \xox{MR}{2916277} \xox{arXiv}{1012.2802}.

\bibitem[MTV07]{MTV}
Marco Mackaay, Paul Turner, and Pedro Vaz, \emph{A remark on {R}asmussen's
  invariant of knots}, J. Knot Theory Ramifications \textbf{16} (2007), no.~3,
  333--344 \xox{MR}{2320159} \xox{arXiv}{math/0509692}.

\bibitem[OSS14]{ossz}
Peter Ozsv\'ath, Andr{\'a}s~I. Stipsicz, and Zolt{\'a}n Szab\'o,
  \emph{Concordance homomorphisms from {K}not {F}loer homology}, 2014
  \xox{arXiv}{1407.1795}.

\bibitem[PP87]{przytycka}
Teresa Przytycka and J{\'o}zef~H. Przytycki, \emph{Signed dichromatic graphs of
  oriented link diagrams and matched diagrams}, 1987 notes, Univ. of British
  Columbia.

\bibitem[Ras10]{Ras1}
Jacob~A. Rasmussen, \emph{Khovanov homology and the slice genus}, Invent. Math.
  \textbf{182} (2010), no.~2, 419--447 \xox{MR}{2729272}
  \xox{arXiv}{math/0402131}.

\bibitem[RW15]{folklore}
David E.~V. Rose and Paul Wedrich, \emph{Deformations of colored
  $\mathfrak{sl}_n$ link homologies via foams}, 2015 \xox{arXiv}{1501.02567}.

\bibitem[Wu09]{Wu1}
Hao Wu, \emph{On the quantum filtration of the {K}hovanov-{R}ozansky
  cohomology}, Adv. Math. \textbf{221} (2009), no.~1, 54--139 \xox{MR}{2509322}
  \xox{arXiv}{math/0612406}.

\bibitem[Wu11]{genericwu}
\bysame, \emph{Generic deformations of the colored
  {$\mathfrak{sl}(N)$}-homology for links}, Algebr. Geom. Topol. \textbf{11}
  (2011), no.~4, 2037--2106 \xox{MR}{2826932} \xox{arXiv}{1011.2254}.

\bibitem[Wu12]{wuequivariant}
\bysame, \emph{Equivariant {K}hovanov-{R}ozansky homology and {L}ee-{G}ornik
  spectral sequence}, 2012 \xox{arXiv}{1211.6732} to appear in Quantum Topol.

\end{thebibliography}
\end{document}